\documentclass[12pt,plain]{article}
\usepackage{fancyhdr}
\usepackage{verbatim}
\usepackage{indentfirst}
\usepackage{graphicx}
\usepackage{color}
\usepackage{newlfont}
\usepackage{amssymb}
\usepackage[intlimits]{amsmath}
\usepackage{latexsym}
\usepackage{amsthm}
\usepackage{amscd}
\usepackage{fullpage}
\usepackage{setspace}
\usepackage{multirow}
\usepackage{hyperref}
\usepackage[dvips, hmargin=2.1cm,vmargin=3cm]{geometry}
\usepackage{multirow}

\newcommand{\de}{\mathrm{d}}

\newcommand{\R}{\mathbb{R}}
\newcommand{\C}{\mathbb{C}}

\newcommand{\tha}{\tfrac{1}{2}}

\newcommand{\be}{\begin{equation}}
\newcommand{\bey}{\begin{eqnarray}}
\newcommand{\ee}{\end{equation}}
\newcommand{\eey}{\end{eqnarray}}
\newcommand{\ba}{\begin{array}}
\newcommand{\ea}{\end{array}}

\newcommand{\bm}[1]{\mbox{\boldmath{$#1$}}}

\theoremstyle{plain}                    
\newtheorem{theorem}{Theorem}[section]
\newtheorem{lem}[theorem]{Lemma}            
\newtheorem{prop}[theorem]{Proposition}
\newtheorem{cor}[theorem]{Corollary}
\theoremstyle{definition}

\newtheorem{remark}[theorem]{Remark}

\input xy
\xyoption{all}

\begin{document}
\title{Smooth Sums over Smooth $k$-Free Numbers\\ and Statistical Mechanics 
}
\author{Francesco Cellarosi$^{*}$ 
}
\maketitle

\let\oldthefootnote\thefootnote
\renewcommand{\thefootnote}{\fnsymbol{footnote}}
\footnotetext[1]{University of Illinois, Urbana-Champaign, IL, U.S.A. Email: \texttt{fcellaro@illinois.edu}}
\let\thefootnote\oldthefootnote

\begin{abstract}
We provide an asymptotic estimate for certain sums over $k$-free integers with small prime factors. These sums depend upon a complex parameter $\alpha$ and involve a smooth cut-off $f$. They are a variation of several classical number-theoretical sums. One term in the asymptotics is an integral operator whose kernel is the $\alpha$-convolution of the Dickman-de Bruijn distribution, and the other term is explicitly estimated. The trade-off between the value of $\alpha$ and the regularity of $f$ is discussed. This work generalizes the results of \cite{Cellarosi-Sinai-Mobius-2, Cellarosi-Sinai-Mobius-1}, where $k=2$ and $\alpha=1$.
\end{abstract}


\section{Introduction}
The study of the typical behavior of arithmetic functions has a long history in number theory. Let $n$ denote a positive integer. 
Let $\omega(n)$ (resp. $\Omega(n)$) denote the number of prime divisors of $n$, counted without (resp. with) multiplicity.
If $d(n)$ denotes the number of divisors of $n$, then clearly $\Omega(n)\leq\log n/\log 2$, and $2^{\omega(n)}\leq d(n)\leq 2^{\Omega(n)}\leq n$. Notice that $2^{\omega(n)}$ equals the number of square-free divisors of $n$.

We are interested in $k$-free numbers, i.e. integers such that $p^k\nmid n$ for every prime $p$. Notice that for $k=2$ the set of square-free numbers is characterized by the condition $\Omega(n)=\omega(n)$. 
This paper is devoted to the study of the asymptotic behavior of certain sums over $k$-free numbers, with an additional restriction on the size of their prime factors.
Let us fix an integer $k\geq2$, $\alpha\in\C$, and a bounded function $f:\R\to\C$. Define the sum 
\bey
S_{\Omega,f}(k,\alpha;N)&=&\sum_{\scriptsize{\ba{c} 
\mbox{$n$ $k$-free}\\ p|n\Rightarrow p\leq N \ea}}f\!\left(\frac{\log n}{\log N}\right)\frac{\alpha^{\Omega(n)}}{n}.\label{two-main-sums-1}
\eey

We shall refer to (\ref{two-main-sums-1}) as a \emph{smooth} sum because the regularity of $f$ plays an important role in our analysis. The $k$-free numbers involved in the sum classically called \emph{smooth} because their prime factors are considerably smaller then the numbers themselves.

We shall always assume that $\alpha\neq0$, otherwise the sum (\ref{two-main-sums-1}) is identically zero.
Let us define the function space 
\be\mathcal S_\eta(\R)=\left\{f:\R\to\C 
:\: \exists  
C>0\:\: \mbox{s.t.}\:\: 
\left|\hat f(\lambda)\right|\leq \frac{C}{1+|\lambda|^\eta}\:\:\forall\lambda\in\R\right\},\nonumber\ee 
where $\hat f$ is the Fourier transform of $f$. 
Notice that, for example, $\textbf{1}_{[0,1]}\in \mathcal S_1$ and that the Schwartz space $\mathcal S\subset \mathcal S_\eta$ for every $\eta$.
The main result of our paper is the following
\begin{theorem}\label{thm-1}
Let $|\alpha|<2$ and let $f\in \mathcal S_\eta$ with $\eta >|\alpha|-\Re \alpha+1$. Then
there exists a non-zero constant $C=C(k,\alpha)\in\C$ such that, for every $R=R(N)$ satisfying 
\be \frac{R}{\log N}\to0\hspace{.5cm}\mbox{and}\hspace{.5cm} \frac{(\log N)^{|\alpha|-\Re\alpha}}{R^{\eta-1}}\to0\hspace{.5cm}\mbox{as $N\to\infty$},\nonumber 
\ee
we have
\be
S_{\Omega,f}(k,\alpha;N)=C\cdot(\log N)^\alpha\cdot \left(\,\int_{|\lambda|\leq R} \varphi^{(\alpha)}(\lambda)\hat f(\lambda)\de\lambda+\varepsilon_N 
\right),\label{statement-thm-1}
\ee
where 
\be
\varphi^{(\alpha)}(\lambda)=\exp\left\{\alpha\int_0^1 \frac{e^{i\lambda v}-1}{v}\de v\right\}\label{statement-thm-2}
\ee and $\varepsilon_N=\varepsilon_N(k,\alpha,f,R)\to0$ as $N\to\infty$ satisfies
\be\varepsilon_N=O\!\left(\frac{\log\log N}{\log N}\right)+O\!\left(\frac{(\log N)^{|\alpha|-\Re\alpha}}{R^{\eta-1}}\right).\label{epsilon-in-main-theorem}\ee
The constants implied by the $O$-notation in (\ref{epsilon-in-main-theorem}) depend on $k$, $\alpha$, $f$, and $R$.
\end{theorem}

\begin{remark}
The function $\varphi^{(\alpha)}(\lambda)$ in (\ref{statement-thm-2}) is the characteristic function (inverse Fourier transform) of the $\alpha$-convolution of the Dickman-de Bruijn probability distribution on $\R_{\geq0}$. Notice that $\varphi^{(\alpha)}(\lambda)$ need not be bounded.
\end{remark}
\begin{remark}
The integral in (\ref{statement-thm-1}) is $O(1)$. However --depending on the function $f$-- it may tend to zero as $N\to\infty$ and, if this is the case, it may be dominated by the error term $\varepsilon_N$. In Section \ref{section-example} we discuss a concrete example where the integral term is bounded away from zero.
A recent preprint by M. Avdeeva, D. Li and Ya.G. Sinai \cite{Avdeeva-Li-Sinai} gives new information on the integral term in (\ref{statement-thm-1}) when $\alpha$ is a negative real number.
\end{remark}
\begin{remark}
Our methods allow us to enlarge the set of $\alpha\in\C$ for which Theorem \ref{thm-1} holds, provided we modify (\ref{two-main-sums-1}) by considering the sum only over the $k$-free integers that are not divisible by a finite number of primes. For example, (\ref{statement-thm-1}) holds for smooth sums over smooth \emph{odd} $k$-free integers for $|\alpha|<3$.
\end{remark}

Theorem \ref{thm-1} shows that there is a competition between the magnitude of $R(N)$ and the regularity parameter $\eta$ for the function $f$. Two natural choices for $R(N)$ are $\log N/\log\log N$ and $(\log N)^\tau$, $0<\tau<1$, and are covered by the following corollaries.

\begin{cor}\label{cor-1}
Let $R(N)=\frac{\log N}{\log\log N}$ in Theorem \ref{thm-1}. Then
\be
\varepsilon_N=\begin{cases}
O\!\left(\frac{\log\log N}{\log N}\right),&\mbox{if $\eta>|\alpha|-\Re\alpha+2$};\\
O\!\left(\frac{(\log\log N)^{\eta-1}}{(\log N)^{\eta-(|\alpha|-\Re\alpha+1)}}\right),&\mbox{if $|\alpha|-\Re\alpha<\eta\leq |\alpha|-\Re\alpha+2$}.
\end{cases}\nonumber
\ee
\end{cor} 

\begin{cor}\label{cor-2}
Let $R(N)=(\log N)^{1-\tau}$ in Theorem \ref{thm-1}. Then
\be
\varepsilon_N=\begin{cases}
O\!\left(\frac{\log\log N}{\log N}\right),&\mbox{if $0<\tau\leq\frac{\eta-(|\alpha|-\Re\alpha+2)}{\eta-1}$ -- case (a)};\\
O\!\left((\log N)^{-\eta(1-\tau)-\tau+(|\alpha|-\Re\alpha+1)}\right),&\mbox{if $0\vee\frac{\eta-(|\alpha|-\Re\alpha+2)}{\eta-1}<\tau<\frac{\eta-(|\alpha|-\Re\alpha+1)}{\eta-1}$ -- case (b)}.
\end{cases}\nonumber
\ee
\end{cor} 
The two cases (a) and (b) are summarized in Figure \ref{fig-two-regions}, where the trade-off between $\tau$ and $\eta$ is apparent.
\begin{figure}[h!]
\begin{center}
\includegraphics[width=15cm]{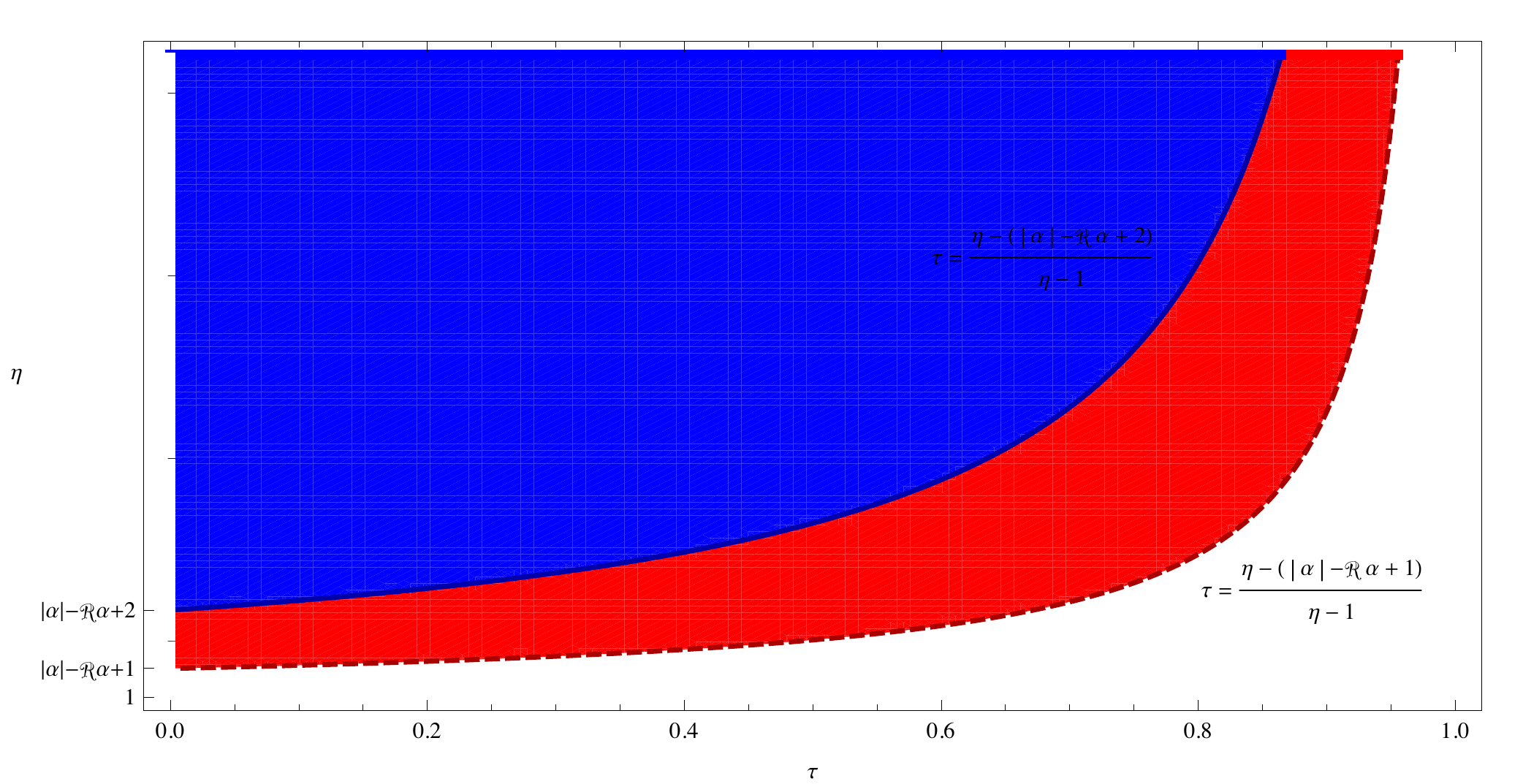}
\caption{Plot of the two regions (a) (blue) and (b) (red), wherein the error terms $O\!\left(\frac{\log\log N}{\log N}\right)$ and $O\!\left((\log N)^{-\eta(1-\tau)-\tau+(|\alpha|-\Re\alpha+1)}\right)$ in Corollary \ref{cor-2} 
dominate respectively. The boundary of the two regions are hyperbol\ae, whose equations are also shown.}\label{fig-two-regions}
\end{center}
\end{figure}

The paper is organized as follows. Section \ref{section-motivation} serves as context and motivation; it includes results concerning the average of certain arithmetic functions, some asymptotic results about smooth numbers, and convolutions of the Dickman-De Bruijn distribution. In Section \ref{section-the-ensemble} we introduce a complex measure and we rewrite the sum (\ref{two-main-sums-1}) by means of a random\footnote{We borrow the classical probabilistic terminology, although for most values of the parameter $\alpha$, we are not dealing with a probability measure.} variable, following ideas from Statistical Mechanics already used in \cite{Cellarosi-Sinai-Mobius-1, Cellarosi-Sinai-Mobius-2}. The main results in Section \ref{section-limit-thm} are Theorem \ref{pw-convergence-varphi_xi_N}, that is devoted to the pointwise approximation of the characteristic function of the above random variable, and Proposition \ref{prop-varphi-tau}, dealing with some integral estimates. Theorem \ref{thm-1} follows from these results.
The proof of Theorem \ref{pw-convergence-varphi_xi_N} occupies Section \ref{section-proof-of-theorem-pw...} and Appendix A.
Section \ref{section-example} discusses a concrete example where the function $f$ is $C^\infty$ with compact support.
\section*{Acknowledgments}
The author is grateful to Yakov G. Sinai, Dong Li, and Maria Avdeeva for useful discussions.
He also acknowledges the support of the Mathematical Sciences Research Institute in Berkeley CA, where most of this work was carried out in the Spring 2012 during the \emph{Random Spatial Processes} special program.
The author also wishes to thank A.J. Hildebrand for pointing out the references \cite{Tenenbaum-Wu-I, Hanrot-Tenenbaum-Wu-II, Tenenbaum-Wu-III, Tenenbaum-Wu-IV} to him, and for suggesting a way to improve the error term (\ref{epsilon-in-main-theorem}) in Theorem \ref{thm-1}.

\section{Motivation}\label{section-motivation}

\subsection{Normal and Average Orders or Arithmetic Functions}
Hardy and Ramanujan \cite{MR2280878} proved that the \emph{normal order} of $\omega(n)$ and $\Omega(n)$ is $\log\log n$, i.e.
for every $\varepsilon>0$, the set of $n\leq x$ such that the inequalities
\be (1-\varepsilon)\log\log n\leq \omega(n)\leq(1+\varepsilon)\log\log n\nonumber\ee
fail to hold has cardinality $o(x)$ as $x\to\infty$ (and the same statement is true for $\Omega(n)$ in place of $\omega(n)$).
Erd\H{o}s and Kac \cite{Erdos-Kac-1940} improved on this result and established a Central Limit Theorem:
\be 
\frac{1}{x}\left|\left\{n\leq x:\: a\leq\frac{\omega(n)-\log\log n}{\sqrt{\log\log n}}\leq b\right\}\right|=\frac{1}{\sqrt{2\pi}}\int_a^be^{-t^2/2}\de t+O\left(\frac{1}{\sqrt{\log\log x}}\right)\label{Erdos-Kac}.\ee

Dirichlet proved that the \emph{average order} order of $d(n)$ is $\log n$, i.e. $\sum_{n\leq x}d(n)\sim x\log x$. More precisely
\be \sum_{n\leq x}d(n)=x\log x+(2\gamma-1)x+\Delta(x),\nonumber\ee
where $\gamma$ is Euler-Mascheroni's constant and $\Delta(x)=O(x^{1/2})$. Finding the smallest $\theta>0$ such that $\Delta(x)=O(x^{\theta+\varepsilon})$ for every $\varepsilon>0$ is known as the \emph{Dirichlet divisor problem}. Several authors have improved Dirichlet's bound $\theta\leq \frac{1}{2}$ towards the conjectured valued $\theta=\frac{1}{4}$, the current record being $\theta\leq\frac{131}{416}$ (Huxley \cite{Huxley-1992}).

Mertens \cite{Mertens-1874-ein, Mertens-1874-ueber} proved that the average order of $2^{\omega(n)}$ is $\frac{1}{\zeta(2)}\log n$. More precisely
\be \sum_{n\leq x}2^{\omega(n)}=\frac{1}{\zeta(2)}x\log x+\left(\frac{2\gamma-1}{\zeta(2)}-\frac{2\zeta'(2)}{\zeta^2(2)}\right)x+\Delta^{(2)}(x),\label{average-order-2^omega-Mertens}\ee
where $\Delta^{(2)}(x)=O(x^{1/2}\log x)$.  It is also conjectured that $\Delta^{(2)}(x)=O(x^{\theta+\varepsilon})$ with $\theta=\frac{1}{4}$, and the best known result (assuming the Riemann Hypothesis) is $\theta\leq \frac{4}{11}$ (Baker \cite{MR1397933}).

Grosswald \cite{MR0074459} proved that the average order of $2^{\Omega(n)}$ is $A \log^2 n$, where $A
\approx0.27317$ is an explicit constant. More precisely
\be \sum_{n\leq x}2^{\Omega(n)}=A x\log^2 x+B x\log x+O(x),\label{average-order-2^Omega-Grosswald}\ee
where 
 $B$ is another explicit constant, and the error term is optimal (Bateman \cite{MR0091970}).
 
 It is interesting to generalize the sums (\ref{average-order-2^omega-Mertens}) and (\ref{average-order-2^Omega-Grosswald}) by replacing $2$ by an arbitrary complex number $z$. Moreover, one can restrict the summation to the set of square-free integers (characterized by the condition $\omega(n)=\Omega(n)$) by multiplying the summand by $\mu^2(n)$, where $\mu$ is the M\"obius function. Selberg \cite{MR0067143} proved that for every $z\in\C$, as $x\to\infty$
\bey
\sum_{n\leq x}z^{\omega(n)}&=&F(z)\,x(\log x)^{z-1}+O\left(x(\log x)^{\Re z-2}\right),\label{Selberg-1}\\
\sum_{n\leq x}\mu^2(n) z^{\omega(n)}&=&G(z)\,x(\log x)^{z-1}+O\left(x(\log x)^{\Re z-2}\right),\label{Selberg-2}
\eey
and, under the assumption that $|z|<2$,
\bey
\sum_{n\leq x}z^{\Omega(n)}&=&H(z)\,x(\log x)^{z-1}+O\left(x(\log x)^{\Re z-2}\right),\label{Selberg-3}
\eey
where
\bey
F(z)&=&\frac{1}{\Gamma(z)}\prod_p\left(1+\frac{z}{p-1}\right)\left(1-\frac{1}{p}\right)^z,\nonumber\\
G(z)&=&\frac{1}{\Gamma(z)}\prod_p\left(1+\frac{z}{p}\right)\left(1-\frac{1}{p}\right)^z,\nonumber\\
H(z)&=&\frac{1}{\Gamma(z)}\prod_p\left(1-\frac{z}{p}\right)^{-1}\left(1-\frac{1}{p}\right)^z.\label{H}
\eey
The convergence in (\ref{Selberg-1})-(\ref{Selberg-3}) with respect to $z$ is uniform on compact sets.
Notice that (\ref{Selberg-1}) yields the first two terms in (\ref{average-order-2^omega-Mertens}), and that the condition $|z|<2$ in (\ref{Selberg-3}) can not be relaxed since, for example, for $z=2$, (\ref{Selberg-3}) and (\ref{average-order-2^Omega-Grosswald}) are different. These results were further improved and generalized by Delange \cite{MR0289432}. He proved the following
\begin{theorem}\label{thm-Delange}
Let $f$ be a non-negative, integer-valued, additive function such that $f(p)=1$, and let $\chi$ be a bounded, multiplicative function such that $\chi(p)=1$. For $\varrho\geq0$ let 
\be \sigma_0(\varrho)=\inf\left\{\sigma>\frac{1}{2}:\: \sum_{k\geq 2}\sum_{p}\frac{|\chi(p^k)|\varrho^{f(p^k)}}{p^{k\sigma}}<\infty\right\}\nonumber\ee
(let us set $\inf\varnothing=+\infty$). Let $E=\{\varrho\geq0:\: \sigma_0(\varrho)<1\}$ and set $R=\sup E\geq1$. Then there exists a sequence $\left\{A_j(z)\right\}_{j\geq 0}$ of holomorphic functions on $|z|< R$ (and continuous on $|z|\leq R$ if $R\in E$) such that for every $q\geq 0$
\be\sum_{n\leq x}\chi(n)z^{f(n)}=x(\log x)^{z-1}\left(\sum_{j=0}^q\frac{A_j(z)}{(\log x)^j}+O\left(\frac{1}{(\log x)^{q+1}}\right)\right)\hspace{.5cm}\mbox{as $x\to\infty$}\label{Delange}.\ee
The constant implied by the $O$-notation is uniform in $z$ on compact sets.
\end{theorem}
In other words, one can write an asymptotic expansion for the sum in (\ref{Delange}) in powers of $\log x$ to arbitrary order. 
The functions $A_j(z)$ in (\ref{Delange}) can be constructed explicitly, in particular 
\be A_0(z)=\frac{1}{\Gamma(z)}\prod_{p}\left(1+\sum_{k=1}^{\infty}\frac{\chi(p^k)z^{f(p^k)}}{p^k}\right)\left(1-\frac{1}{p}
\right)^z\label{A_0(z)}.\nonumber\ee 
Notice that the results by Selberg follow from Theorem \ref{thm-Delange}, namely  $(f,\chi)=(\omega,1)$ gives (\ref{Selberg-1}), $(f,\chi)=(\omega,\mu^2)$ gives (\ref{Selberg-2}), and $(f,\chi)=(\Omega,1)$ gives (\ref{Selberg-3}).
Theorem \ref{thm-Delange} implies a general result concerning the average order of $\chi$ along the level sets of $f$:
\begin{theorem}\label{thm-Delange-appl}
Let $f,\chi$ be as in Theorem \ref{thm-Delange}. For every $m\geq 1$, there exists a sequence $\{P_j\}_{j\geq0}$ of polynomials of degree $\leq m-1$ such that for every $q\geq 0$
\be\sum_{\scriptsize{\ba{c}n\leq x\\ f(n)=m\ea}}\chi(n)=\sum_{j=1}^{q}\frac{x P_j(\log\log x)}{(\log x)^{j+1}}+O\left(\frac{x(\log\log x)^{m-1}}{(\log x)^{q+2}}\right)\hspace{.5cm}\mbox{as $x\to\infty$}.\nonumber\ee 
The coefficient of the monomial of degree $m-1$ in $P_j$ is 
\be\frac{(-1)^j}{(m-1)!}\left.\frac{\de^j}{\de s^j}\left(\frac{1}{s}\prod_p\left(1+\sum_{f(p^k)=0}\frac{\chi(p^k)}{p^{ks}}\right)\right)\right |_{s=1}.\nonumber\ee
\end{theorem}
In particular Theorem \ref{thm-Delange-appl} implies the results by Landau \cite{Landau-1911}: as $x\to\infty$
\bey
\sum_{\scriptsize{\ba{c}n\leq x\\ \omega(n)=m\ea}}1\sim \sum_{\scriptsize{\ba{c}n\leq x\\ \Omega(n)=m\ea}}1\sim\sum_{\scriptsize{\ba{c}n\leq x\\ \omega(n)=\Omega(n)=m\ea}}1\sim\frac{x}{\log x} \frac{(\log\log x)^{m-1}}{(m-1)!}\nonumber
\eey
by taking $(f,\chi)=(\omega,1)$, $(f,\chi)=(\Omega,1)$, and $(f,\chi)=(\omega,\mu^2)$ respectively. 
Let us also point out that the error terms coming from Theorem \ref{thm-Delange-appl} (i.e. $O\!\left(\frac{x(\log\log x)^{m-1}}{(\log x)^2}\right)$) are better than the ones given by Landau (i.e. $O\!\left(\frac{x(\log\log x)^{m-2}}{\log x}\right)$). Recall that the integers for which $\Omega(n)=m$ are called \emph{$m$-almost primes}.
\subsubsection{Ordinary and Logarithmic Mean Values}
Given an arithmetic function $g$, one can define the \emph{(ordinary) mean value of $g$} as
$M[g]=\lim_{x\to\infty}\frac{1}{x}\sum_{n\leq x}g(n),$
and the \emph{logarithmic mean value of $g$} as
$L[g]=\lim_{x\to\infty}\frac{1}{\log x}\sum_{n\leq x}\frac{g(n)}{n}.$
It is a classical fact that, if $M[g]$ exists, then $L[g]$ also exists and they are equal. The other implication is in general not true. For example, $M[\mu]=L[\mu]=0$ and $M[\mu^2]=L[\mu^2]=\tfrac{6}{\pi^2}$. 

The sums (\ref{two-main-sums-1}) we are concerned with can be seen as partial sums for some weighted logarithmic averages.

\subsection{Sums over Smooth Integers}\label{section-sums-over-smooth-integers}
Sums of type (\ref{Selberg-1}-\ref{Delange}) can be further generalized by setting a constraint on the size of primes in the factorization of $n$. Integers whose prime factors are all $\leq y$ are called \emph{$y$-smooth}, while those whose prime factors are $\geq c$ are called \emph{$c$-rough} or \emph{$c$-jagged}.
Generalizing (\ref{Selberg-3}), we can introduce the sum

\be \Psi(z;x,y,c)=\sum_{\scriptsize{\ba{c} n \leq x\\ p|n\Rightarrow c<p\leq y\ea}}z^{\Omega(n)}.\label{general-Psi}\ee
Clearly, if $c<2$ we set no restriction on the roughness of the integers in the sum, and if $y\geq x$ we set no restriction for their smoothness. The inequality $|z|<2$ in (\ref{Selberg-3}) can be replaced by $|z|<c$ for $c\geq 2$, and (using the notation (\ref{general-Psi}) we just introduced) we have
\be \Psi(z;x,x,c)=H_c(z)\, x(\log x)^{z-1}+O(x(\log x)^{\Re z-2})\hspace{.5cm}\mbox{as $x\to\infty$},\label{Psi(z,x,x,c)}\ee
where \be H_c(z)=\frac{1}{\Gamma(z)}\prod_{p\leq c}\left(1-\frac{1}{p}\right)^z\prod_{p>c}\left(1-\frac{z}{p}\right)^{-1}\left(1-\frac{1}{p}\right)^z,\hspace{.5cm}\mbox{cfr. (\ref{H})}.\label{Hc}\ee
Moreover, the results by Delange (e.g. Theorem \ref{thm-Delange}) allow to write the asymptotic of $\Psi(z;x,x,c)$ in powers of $\log x$, where the functions $z\mapsto A_j(z)$ (see (\ref{Delange})) are holomorphic for $|z|<p'$, where $p'=p'(c)$ is the least prime larger than $c$.

For $z=1$ we have the well-known counting of $y$-smooth integers:

\be \Psi(x,y)=\Psi(1;x,y,1)=|\{n\leq x:\: \mbox{$n$ is $y$-smooth}\}|\label{Psi-smooth}.\ee

Dickman \cite{Dickman-1930} proved that $\Psi(x,y)\sim x\rho(u)$ as $x\to\infty$ when for $y=x^{1/u}$ for some $u\geq1$, where $\rho$ is the Dickman-de Bruijn function, i.e. the solution of the delay differential equation $u\rho'(u)+\rho(u-1)=0$ with initial condition $\rho(u)=1$ for $0\leq u\leq 1$. The range of $u$'s for which the above asymptotic result is valid has been significantly enlarged, and explicit error terms are known. Namely
\be \Psi(x,y)=x\rho(u)\left(1+O\!\left(\frac{\log(u+1)}{\log y}\right)\right),\hspace{.5cm}\mbox{where $y=x^{1/u}$},\label{Psi-asymptotic}\ee 
and $y>\exp((\log x)^{5/8+\varepsilon})$ (de Bruijn \cite{deBruijn-1951a, deBruijn-1951b, deBruijn-1966}), or $y>\exp((\log\log x)^{5/3+\varepsilon})$ (Hildebrand \cite{Hildebrand-1984a}). Moreover, (\ref{Psi-asymptotic}) holds uniformly for $y\geq (\log x)^{2+\varepsilon}$ if and only if the Riemann Hypothesis is true (Hildebrand \cite{Hildebrand-1984a}).
For smaller values of $y$ the asymptotic result (\ref{Psi-asymptotic}) is not true anymore. One has, for example,
\be \Psi(x,\log^A x)=x^{1-1/A+O(1/\log\log x)}\hspace{.5cm}\mbox{for $A>1$}\nonumber\ee
(Granville \cite{Granville2008}, p. 291) and
\be \log\Psi(x,\kappa\log x)=\left(\log(1+\kappa)+\kappa\log(1+1/\kappa)\right)\frac{\log x}{\log\log x}\left(1+O\left(\frac{1}{\log\log x}\right)\right)\nonumber\ee 
(Granville \cite{Granville2008}, Erd\H{o}s \cite{Erdos-1963}).

One can further restrict the sums (\ref{general-Psi}-\ref{Psi-smooth}) to square-free integers (for which $\omega(n)=\Omega(n)$) (cfr. (\ref{Selberg-2})). If $y$ is not too small compared to $x$ (namely $\log y/\log\log x\to\infty $), then the rhs of (\ref{Psi-smooth}) is simply multiplied by the density of square-free numbers, i.e. $6/\pi^2$. 
In the case we shall consider this will not be the case, since we will deal with $y\sim \log n$.

It is also worthwhile to mention that Alladi \cite{MR913766}, Hensley \cite{MR1018380}, and Hildebrand \cite{MR871170,  MR921088} proved an analog of the Erd\"os-Kac result (\ref{Erdos-Kac}) for $y$-smooth integers, with the same mean and variance ($\sim \log\log x$) for $u=\log x/\log y=o(\log\log x)$, while the mean is $\sim u$ and the variance is $\sim u/(\log u)^2$ whenever $\log x\ll y\ll \exp((\log x)^{1/21})$.

A large amount of work on averages of multiplicative functions over smooth integers has been done by G. Tenenbaum and J. Wu \cite{Tenenbaum-Wu-I, Tenenbaum-Wu-III, Tenenbaum-Wu-IV} and G. Hanrot, Tenenbaum and Wu \cite{Hanrot-Tenenbaum-Wu-II}. The reader can refer to  \cite{Tenenbaum-Wu-I} for an historical introduction to the subject.

\subsection{Convolutions of the Dickman-de Bruijn Distribution}

The case of general $\Psi(z,x,y,c)$ has been studied by DeKoninck and Hensley \cite{MR559011}. They proved an approximation of $\Psi$ by means of another function $\psi$ which, in turn, is close to $x(\log x)^{z-1}\rho_z(u)$, where $\rho_z$ is a close relative of the Dickman-de Bruijn function. 

Let us consider $\alpha\geq 1$ and let us use the notation $\Psi_\alpha(x,y)=\Psi(\alpha;x,y,\alpha)$. The asymptotic of $\Psi_\alpha(x,x)$  follows from Theorem \ref{thm-Delange}, see (\ref{Psi(z,x,x,c)}-\ref{Hc}). The first result for $y$-smooth numbers is due to Hensley \cite{MR741951}, who gave the asymptotic of $\Psi_\alpha(x,x^{1/u})$, similarly to (\ref{Psi-asymptotic}).
He proved that for every $\alpha\geq 1$ and every $0<\varepsilon<1$,
\be \Psi_\alpha(x,x^{1/u})=\rho_{\alpha}(u)\,x(\log x)^{\alpha-1}(1+o(1)),\nonumber\ee 
uniformly in $1\leq u\leq (1-\varepsilon)\log\log x/\log\log\log x$ as $x\to\infty$,
where $\rho_\alpha(u)$ satisfies the delay differential equation involving the function $A_0$ (see (\ref{A_0(z)})) from Theorem \ref{thm-Delange}.
\be
\begin{cases}
\rho_\alpha(u)=0&
u<0
\\
\rho_\alpha(u)=A_0(\alpha) & 0\leq u\leq 1\\
-u^{\alpha}\rho_\alpha'(u)=\alpha(u-1)^{\alpha-1}\rho_\alpha(u-1)&u>1.
\end{cases}\label{rho-alpha}
\ee
For $\alpha=1$ the function (\ref{rho-alpha}) agrees with the Dickman-de Bruijn function (see Section \ref{section-sums-over-smooth-integers}). It is convenient to introduce a probability distribution  on $\R_{\geq0}$, whose density $w_\alpha(u)$ satisfies,
\be
w_\alpha(u)=\frac{e^{\alpha \gamma}}{\Gamma(\alpha)}u^{\alpha-1}\rho_{\alpha}(u).\nonumber 
\ee
Notice that for $\alpha=1$ the two functions $\rho(u)$ and $w_1(u)$ differ by a multiplicative constant, but for general $\alpha$ this is not the case. The density $w_\alpha(u)$ decays faster than exponentially 
as $u\to\infty$ (see \cite{deBruijn-1951a}). 
It is known that the characteristic function $\varphi^{(\alpha)}$ of $w_\alpha$ is given by
\be\varphi^{(\alpha)}(\lambda)=\int_{0}^\infty e^{i\lambda u}w_\alpha(u)\de u=\exp\left(\alpha\int_0^1 \frac{e^{i\lambda v}-1}{v}\de v\right)\label{varphi-alpha}.\ee
In other words, since $\varphi^{(\alpha)}(\lambda)=(\varphi^{(1)}(\lambda))^\alpha$, $w_\alpha$ is the density of the \emph{$\alpha$-convolution of the Dickman-de Bruijn distribution}.
Later we shall use the fact that 
\be\varphi^{(1)}(\lambda)\sim\frac{ie^{-\gamma}}{\lambda}\hspace{.3cm}\mbox{as $|\lambda|\to\infty$}\label{label-asymptotic-decay-varphi(lambda)},
\ee
where $\gamma$ is Euler-Mascheroni's constant.
In our analysis, we shall consider $\alpha\in\C$. In this case $\alpha$-convolutions of the Dickman-de Bruijn distribution cannot be considered as probability distributions on $\R_{\geq0}$, but only as \emph{distributions in the sense of Schwartz}. More precisely, they will be elements of $(\mathcal S_\eta(\R))^*$ for $\eta>|\alpha|-\Re\alpha+1$.

\section{Reformulaton of the Problem. The Ensemble $\frak X_N^{(k)}$}\label{section-the-ensemble}
Recall the prime counting function $\pi(N)=|\{p\in\mathcal P:\: p\leq N\}|$. The classical Prime Number Theorem gives $\pi(N)\sim\frac{N}{\log N}$ as $N\to\infty$ 
Let us consider the set 
\be\frak X_N^{(k)}=\left\{x=\prod_{p\leq N} p^{\nu(p)}:\: 0\leq \nu(p)\leq k-1 
\right\}\nonumber\ee 
consisting of all $k$-free integers whose prime factors do not exceed $N$.
Notice that $|\frak X_N^{(k)}|=k^{\pi(N)}$ and $\max\frak X_N^{(k)}=\prod_{p\leq N} p^{k-1}=
e^{(k-1)\pi(N)\log\pi(N)(1+o(1))}$. This means that $\frak X_N^{(k)}$ is a sparse set. For $x\in\frak X_N^{(k)}$ we have $\Omega(x)=\sum_{p\leq N} \nu(p) 
$ and $\omega(x)=\left|\{p\leq N:\: \nu(p)>0\}\right|$. By definition, all $x\in\frak X_N^{(k)}$ are $k$-free, and it easy to check that all $k$-free $x\leq p_{\pi(N)}$ belong to $\frak X_N^{(k)}$.

Let us introduce a 
complex measure $P_\Omega^{(\alpha)}$ 
on $\frak X_N^{(k)}$: for every $X\subseteq\frak X_N^{(k)}$, let 
\be P_\Omega^{(\alpha)}(X)=\sum_{x\in X}\frac{\alpha^{\Omega(x)}}{x}\nonumber
.\ee
For example $\frak X_5^{(2)}=\{1,2,3,5,6,10,15,30\}$, and $P^{(\alpha)}_\Omega(\{1,3,10\})=
\frac{\alpha^0}{1}+\frac{\alpha^1}{3}+\frac{\alpha^2}{10}
$. Another example is $\frak X_{4}^{(3)}=\{1,2,3,4,6,9,18,36\}$, where $P_{\Omega}^{(\alpha)}(\{1,9,18\})=\frac{\alpha^0}{1}+\frac{\alpha^2}{9}+\frac{\alpha^3}{18}
$. 

Using the terminology of Statistical Mechanics, we introduce the \emph{partition function} 
\be Z_{\Omega,N}^{(k,\alpha)}=P_\Omega^{(\alpha)}(\frak X_N^{(k)}).
\nonumber
\ee

We prove an asymptotic result for $Z_{\Omega,N}^{(k,\alpha)}$ as $N\to\infty$. In the case of $\Re \alpha >0$ we have that 
$|Z_N^{(k,\alpha)}|\to\infty$ as $N\to\infty$ and the analogy with Statistical Mechanics suggest the term ``\emph{thermodynamical limit}". 
Let us define a set 
for the parameter $\alpha$ for which the partition function 
vanishes for sufficiently large $N$. This set is responsible for the restriction $|\alpha|<2$ in our Theorem \ref{thm-1}. 
Let
\bey 
\mathcal A_\Omega^{(k)}&=&\left\{z\in\C:\: z=p e^{2\pi i l/k},\:p\in\mathcal P,\:1\leq l\leq k-1\right\}.\nonumber \eey

Notice that $-\mathcal P\subseteq \mathcal A_\Omega^{(k)}$ if $k$ is even and $\mathcal A_\Omega^{(2)}=-\mathcal P$. 
We have the 
following
\begin{lem}\label{lem-asymptotic-ZNkq}
There exist a 
constants $C_\Omega=C_\Omega(k,\alpha)\in\C$ 
such that, as $N\to\infty$,
\bey 
Z_{\Omega,N}^{(k,\alpha)} &=&\begin{cases}0,&\alpha\in\mathcal A_\Omega^{(k)};\\
C_\Omega\, (\log N)^{\alpha}\left(1+ O\!\left(\frac{1}{\log N}
\right)\right),&\mbox{otherwise};
\end{cases}
\label{statement-lem-asymptotic-ZNkq-1}
\eey
where the constant implied by the $O$-notation depends on 
$k$ and $\alpha$.
\end{lem}
\begin{proof}
We can write
\bey
Z_{\Omega,N}^{(k,\alpha)}&=&\prod_{p\leq N}\!\left(1+\frac{\alpha}{p}+\frac{\alpha^2}{p^2}+\ldots+\frac{\alpha^{k-1}}{p^{k-1}}\right).\label{proof-partition-function-Omega-1}\eey
Notice that $\sum_{l=0}^{k-1}(\alpha/u)^l =0$ if and only if $u=\alpha e^{2 p i\frac{l}{k}} $, $l=1,\ldots, k-1$, and there is a prime $p$ of this form if and only if $\alpha\in\mathcal A_\Omega^{(k)}$. 

Now, let $\alpha\notin \mathcal A_\Omega^{(k)}$, $\alpha\neq 0$. Then there exist a constant $d=d(\alpha)$ such that $z(p)=\sum_{l=0}^{k-1}(\alpha/p)^l\in B_1(1/2)=\{z\in\C:\:|z-1|<1/2\}$  
for every $p>d$. For all such $p$'s we can write $\log z(p)=\log |z(p)|+i\arg(z(p))$ and choose the same branch of the logarithm. From (\ref{proof-partition-function-Omega-1}) we get

\bey
Z_{\Omega,N}^{(k,\alpha)}
&=&C_1\cdot \prod_{d<p\leq N}\left(1+\frac{\alpha}{p}\right)\cdot\prod_{d<p\leq N}\left(1+\frac{\alpha^2/p^2+
\ldots+\alpha^{k-1}/p^{k-1}
}{1+\alpha/p 
}
\right),\label{pf-lm-asymptotic-ZN-1} 
\eey
where $C_1=C_1(\alpha)=\prod_{p\leq d}z(p)\neq0$. 
The second factor in (\ref{pf-lm-asymptotic-ZN-1}) gives the asymptotic $(\log N)^\alpha$. In fact, by taking the logarithm, we have
\bey
\log Z_{\Omega,N}^{(k,\alpha)}&=&\log C_1+
\alpha\sum_{d<p\leq N}\frac{1}{p}-\sum_{d<p\leq N}\left(\log\left(1+\frac{\alpha}{p}\right)-\frac{\alpha}{p}\right)+\sum_{d<p\leq N}\log\left(1+\frac{\alpha^2/p^2-\alpha^k/p^k}{1-\alpha^2/p^2}\right)=\nonumber\\
&=&\log C_1+\alpha(\log\log N+C_2)+C_3+C_4+ O\!\left(\frac{1}{\log N}\right),\nonumber
\eey
where $C_2=C_2(\alpha)$, $C_3=C_3(\alpha)$ and $C_4=C_4(k,\alpha)$ do not depend on $N$.
We then see immediately that (\ref{statement-lem-asymptotic-ZNkq-1}) holds 
with $C_\Omega(k,\alpha)=C_1e^{\alpha C_2+C_3+C_4}$. 
\end{proof}

\begin{remark}
A classical theorem by Mertens \cite{Mertens-1874-ein} gives $C_\Omega(2,1)=e^{-\gamma}$. \end{remark}

Let us observe that the sum 
(\ref{two-main-sums-1}) 
can be written as 
\bey
S_{\Omega,f}(k,\alpha;N)=\sum_{x\in\frak X_N^{(k)}}f\left(\frac{\log x}{\log N}\right) P_{\Omega}^{(\alpha)}(\{x\}).\nonumber
\eey
We can write \be \log x=\sum_{p\leq N} \nu(p) \log p\nonumber\ee
and
introduce the function \be\frak X_N^{(k)}\ni x\mapsto \xi_N(x)=\frac{\log x}{\log N}=\sum_{p\leq N}\nu(p)\frac{\log p}{\log N}.\nonumber
\ee

Our main theorem will follow from the study of the distribution of $\xi_N$ with respect to the measures $P_{\Omega}^{(\alpha)}$. 
It is convenient for us to introduce the \emph{normalized} measure 
$ P_{\Omega,N}^{(k,\alpha)}$ 
on $\frak X_N^{(k)}$, i.e. \be P_{\Omega,N}^{(\alpha,k)}=(Z_{\Omega, N}^{(k,\alpha)})^{-1} P_\Omega^{(\alpha)}.\nonumber\ee

For every $p\leq N$
, the measure 
$P_{\Omega,N}^{(k,\alpha)}$ on $\frak X_N^{(k)}$ induce 
a measure
on $\{0,1,\ldots, k-1\}$  
via the function 
\be \frak X_N^{(k)}\ni x\mapsto\nu(p)=\nu(p,x)=\max\{l\geq 0:\: p^l|x\}. 
\ee 
This means that for each $0\leq t\leq k-1$ we are given the `probability' $P_{\Omega,N}^{(k,\alpha)}(\{x\in\frak X_N^{(k)}\colon \nu(p,x)=t\})$. 
We shall use the distributions of the $\nu(p)$'s to compute the ones of $\xi_N$.
We have the following simple
\begin{lem}\label{lem-distri-nu_j}
For every $0\leq t\leq k-1$
\be P_{\Omega,N}^{(k,\alpha)}(\{\nu(p)=t\})=\frac{\alpha^t p^{k-1-t}(p-\alpha)}{p^k-\alpha^k}. 
\nonumber\ee
\end{lem}
\begin{proof}
The result follows from the straightforward computation
\bey
P_{\Omega,N}^{(k,\alpha)}(\{\nu(p)=t\})&=&\frac{1}{Z_{\Omega,N}^{(k,\alpha)}} \frac{\alpha^t}{p^t}\prod_{p'\leq N,\: p'\neq p 
}\!\left(1+\frac{\alpha}{p'}+\ldots+\frac{\alpha^{k-1}}{p'^{k-1}}\right)=\frac{\alpha^t/p^t}{1+\alpha/p+\ldots+\alpha^{k-1}/p^{k-1}}=\nonumber\\
&=&\frac{\alpha^t}{p^t}\frac{1-\alpha/p}{1-\alpha^k/p^k}=\frac{\alpha^t p^{k-1-t}(p-\alpha)}{p^k-\alpha^k},\nonumber
\eey
and the normalization of the measure $P_{\Omega,N}^{(k,\alpha)}$. 
\end{proof}

\begin{remark}
The functions $\nu(p)$ are independent with respect to 
the measure $P_{\Omega,N}^{(k,\alpha)}$ 
, i.e. for every $r\geq 1$, every $p_{j_1}<p_{j_2}<\ldots<p_{j_r}\leq N 
$, and every $(\epsilon_1,\epsilon_2,\ldots,\epsilon_r)\in\{0,1,\ldots,k-1\}^r$, we have
\be P_{\Omega,N}^{(k,\alpha)}\left(\left\{\nu(p_{j_1})=\epsilon_1,\:\nu(p_{j_2})=\epsilon_{2},\ldots,\nu(p_{j_r})=\epsilon_r\right\}\right)=\prod_{l=1}^r P_{\Omega,N}^{(k,\alpha)}(\{\nu(p_{j_l})=\epsilon_l\}),\nonumber\ee
i.e. joint `probabilities' factor completely.
On the other hand, by Lemma \ref{lem-distri-nu_j}, the $\nu(p)$'s are not identically distributed with respect to  
$P_{\Omega,N}^{(k,\alpha)}$. 
Notice, for instance, that $P_{\Omega,N}^{(k,\alpha)}(\{\nu(p_{\pi(N)})=0\})=1-\frac{\alpha}{N}+O\!\left(\frac{1}{N^k}\right)$  as $N\to\infty$.
The dependence of the ``probabilities" $P_{\Omega,N}^{(k,\alpha)}(\{\nu(p)=t\})$ upon $N$ will play a crucial role in the proof of our main theorem.
\end{remark}

Let us define the 
rational functions $F_{\Omega,t}^{(k,\alpha)}\in 
\C(u)$ such that $F_{\Omega,t}^{(k,\alpha)}(p)=P_{\Omega,N}^{(k,\alpha)}(\{\nu(p)=t\})$:
\bey
F_{\Omega,t}^{(k,\alpha)}(u)&=&
\frac{\alpha^t u^{k-1-t}(u-\alpha)}{u^k-\alpha^k}=
\frac{\alpha^t u^{k-1-t}}{u^{k-1}+\alpha u^{k-2}+\alpha^2 u^{k-3}+\ldots+\alpha^{k-2} u+\alpha^{k-1}},
\label{def-F_Omega,t} 
\eey
where $0\leq t\leq k-1$. Notice that the poles of $F_{\Omega, t}^{(k,\alpha)}(u)$ are $- \alpha e^{2\pi i \frac{l}{k}}$, $1\leq l\leq k-1$. 
Let $b_{\Omega,t}^{(k,\alpha)}(l)$, $l\geq0$ the coefficient of the Laurent series for $F_{\Omega,t}^{(k,\alpha)}(u)$ on the neighborhood of infinity 
$|u|> |\alpha| $, 
\be
F_{\Omega,t}^{(k,\alpha)}(u)=\sum_{l=0}^\infty \frac{b_{\Omega,t}^{(k,\alpha)}(l)}{u^{l}}. \nonumber
\ee
Notice that $F_{\Omega,t}^{(k,\alpha)}(u)$ has no positive powers in its expansion at $\infty$ since $\frac{\alpha^tu^{k-1-t}}{u^{k-1}}\sim\frac{\alpha^t}{u^t}$.
Set $b_{\Omega,0}^{(k,\alpha)}(l)=0$ for $l<0$. 
We have the following two simple lemmata:

\begin{lem}\label{structure-Laurent-coeffcients-b_Omega-1}
\be
b_{\Omega,0}^{(k,\alpha)}(l)=\begin{cases}\alpha^l,&l\equiv 0 \mod k;\\ 
-\alpha^l,&l\equiv 1\mod k;\\
0,&l\equiv 2,\ldots,k-1\mod k.\end{cases}
\nonumber\ee
\end{lem}

\begin{lem}\label{structure-Laurent-coeffcients-b_Omega-2}
\bey 
b_{\Omega,t}^{(k,\alpha)}(l)&=&\alpha^t b_{\Omega,0}^{(k,\alpha)}(l-t),\hspace{.5cm}t\geq0.\nonumber 
\eey
\end{lem}

\section{Limit theorem for $\xi_N$}\label{section-limit-thm}
As commonly done in Probability Theory, weak convergence of measures is obtained by showing the pointwise convergence of the corresponding characteristic functions (inverse Fourier transforms). Let us define, for $\lambda\in\R$,  
\be\varphi_{\Omega,N}^{(k,\alpha)}(\lambda)=\mathbb E_{P_{\Omega,N}^{(k,\alpha)}}(e^{i\lambda \xi_N}).
\nonumber\ee 
Notice that the chosen normalizations for $P_{\Omega,N}^{(k,\alpha)}$ 
gives $\varphi_{\Omega,N}^{(k,\alpha)}(0)=1$. Moreover, when $P_{\Omega,N}^{(k,\alpha)}$ is a \emph{probability} measure (i.e. when $\alpha$ is a positive real number), then we have the inequality $|\varphi_{\Omega,N}^{(k,\alpha)}(\lambda)|\leq 1$, but this is in general not true. Instead we have, for $|\alpha|<2$, 
\bey
\left|\varphi_{\Omega,N}^{(k,\alpha)}(\lambda)\right|&\leq& \mathbb E_{|P_{\Omega,N}^{(k,\alpha)}|}1=\frac{Z_{\Omega,N}^{(k,|\alpha|)}}{\left|Z_{\Omega,N}^{(k,\alpha)}\right|}=\frac{C_\Omega(k,|\alpha|)}{|C_\Omega(k,\alpha)|}\frac{(\log N)^{|\alpha|}\left(1+O(1/\log N)\right)}{(\log N)^{\Re\alpha}\left(1+O(1/\log N)\right)}=\nonumber\\
&=&O\!\left((\log N)^{|\alpha|-\Re\alpha}\right),
\label{trivial-estimate-varphi}
\eey
where the constant implied by the $O$-notation depends only on 
$k$, and $\alpha$, uniformly in $\lambda$.

Recall $\varphi^{(\alpha)}(\lambda)$ from (\ref{varphi-alpha}), i.e. the characteristic function  of the $\alpha$-convolution of the Dickman-de Bruijn distribution. 
We prove the following
\begin{theorem}\label{pw-convergence-varphi_xi_N}
Let $|\alpha|<2$, and assume that $\lambda=o(\log N)$ as $N\to\infty$. 
Then 
\be \varphi_{\Omega,N}^{(k,\alpha)}(\lambda)= \varphi^{(\alpha)}(\lambda)\left(1 + O\!\left(\frac{1}{\log N}\right) 
+O\!\left(\varepsilon \log|\varepsilon|\right)
\right),
\nonumber\ee 
where $\varepsilon=\frac{\lambda}{\log N}=o(1)$ as $N\to\infty$. 
The constants implied by the $O$-notation depend only on $k$  
and $\alpha$. 
\end{theorem}
\begin{remark}
The case of $\lambda\to0$ is of no harm in dealing with characteristic functions of normalized measures (recall that $\varphi_{\Omega,N}^{(k,\alpha)}(0)=1$), and the error-term bounds obtained are the same as for $\lambda=O(1)$. The term $O\!\left(\frac{1}{\log N}\right)$ in the theorem above prevents an underestimate of the error term in the (uninteresting) case when $\lambda\to0$, i.e. $\varepsilon=o\!\left(\frac{1}{\log N}\right)$.
The interesting application of the above result is for $|\lambda|\to\infty$ more slowly than $\log N$, and the corresponding error term can be simply written as $O(\varepsilon\log|\varepsilon|)$.
\end{remark}

\begin{remark}
Theorem \ref{pw-convergence-varphi_xi_N} generalizes the main result in \cite{Cellarosi-Sinai-Mobius-1, Cellarosi-Sinai-Mobius-2} (where the case $(k,\alpha)=(2,1)$ is addressed), and provides an explicit error term (as opposed to simply an error term $o(1)$).
\end{remark}
A consequence of Theorem \ref{pw-convergence-varphi_xi_N} is the following
\begin{prop}\label{prop-varphi-tau}
Let us fix $k\geq 2$, 
$|\alpha|<2$, 
$f\in\mathcal S_{\eta}$ with $\eta\geq |\alpha|-\Re\alpha+1$. 
Then for every $R=R(N)$ such that 
\be \frac{R}{\log N}\to0\hspace{.5cm}\mbox{and}\hspace{.5cm}\frac{(\log N)^{|\alpha|-\Re\alpha}}{R^{\eta-1}}\to0\hspace{.5cm}\mbox{as $N\to\infty$}\nonumber\ee
\bey
\int_{\R}\varphi_{\Omega,N}^{(k,\alpha)}(\lambda)\hat f(\lambda)\de\lambda&=&\int_{|\lambda|\leq R}\varphi^{(\alpha)}(\lambda)\hat f(\lambda)\de \lambda + \varepsilon_N,\nonumber
\eey
where $\varepsilon_N=\varepsilon_N(k,\alpha,f,R)\to0$ as $N\to\infty$ satisfies 
\be\varepsilon_N=O\!\left(\frac{\log\log N}{\log N}\right)+O\!\left(\frac{(\log N)^{|\alpha|-\Re\alpha}}{R^{\eta-1}}\right).\nonumber
\ee
\end{prop}

\begin{proof}
For every $R=R(N)$ such that $R(N)=o(\log N)$ as $N\to\infty$, let us write
\bey
\int_{\R}\varphi_{\Omega,N}^{(k,\alpha)}(\lambda)\hat f(\lambda)\de \lambda=\int_{|\lambda|\leq R}+\int_{|\lambda|>R}=\mathcal I_1+\mathcal I_2.\nonumber
\eey
Now we can  use Theorem \ref{pw-convergence-varphi_xi_N}:
\bey
\mathcal I_1&=&\int_{|\lambda|\leq R}\varphi^{(\alpha)}(\lambda)\hat f(\lambda)\de\lambda+O\!\left(\int_{|\lambda|<R}\varphi^{(\alpha)}(\lambda)\frac{\lambda}{\log N}\log\!\left|\frac{\lambda}{\log N}\right|\hat f(\lambda)\de\lambda\right)+\nonumber\\
&&+O\!\left(\int_{|\lambda|\leq R}\varphi^{(\alpha)}(\lambda)\frac{1}{\log N}\hat f(\lambda)\de\lambda\right)=
\mathcal I_{1,1}+\mathcal I_{1,2}+\mathcal I_{1,3}\nonumber.
\eey
Notice that, by (\ref{label-asymptotic-decay-varphi(lambda)}) and the fact that $\varphi^{(\alpha)}(\lambda)=(\varphi^{(1)}(\lambda))^\alpha$, we have $\varphi^{(\alpha)}(\lambda)=O\!\left(\lambda^{-\Re \alpha}\right)$ as $\lambda\to\infty$. We can write
\bey
\mathcal I_{1,2}&=&O\!\left(\int_{|\lambda|\leq 1}\left|\frac{\lambda}{\log N}\log\!\left|\frac{\lambda}{\log N}\right|\right|\de\lambda\right)+O\!\left(\int_{1<|\lambda|\leq R} 
\frac{\lambda^{1-\Re \alpha-\eta}}{\log N}\log\left|\frac{\lambda}{\log N}\right|\de\lambda\right)=\nonumber\\
&=&O\!\left(\frac{\log\log N}{\log N}\right)+
\begin{cases}O\!\left(\frac{\log R}{\log N}\log\!\left(\frac{\log^2 N}{R}\right)\right),&\Re\alpha+\eta=2;\\ 
O\!\left(\frac{\log\log N}{\log N}\right)+O\!\left(\frac{R^{2-\Re\alpha-\eta}}{\log N}\log\!\left(\frac{\log N}{R}\right)\right),&\mbox{otherwise};
\end{cases}\nonumber\\
\mathcal I_{1,3}&=&O\!\left(\frac{1}{\log N}\right)+O\!\left(\int_{1<\lambda\leq R}\frac{\lambda^{-\Re\alpha-\eta}}{\log N}\de \lambda\right)=\begin{cases}O\!\left(\frac{\log R}{\log N}\right),&\Re\alpha+\eta=1;\\O\!\left(\frac{R^{1-\Re\alpha-\eta}}{\log N}\right), &\mbox{otherwise}.\end{cases}\nonumber
\eey
Let us observe that $\mathcal I_{1,3}=O(\mathcal I_{1,2})$. To estimate $\mathcal I_2$ we use the trivial estimate (\ref{trivial-estimate-varphi}) and the fact that $\eta>1$:
\bey
\mathcal I_2&=&O\!\left((\log N)^{|\alpha|-\Re\alpha}\int_{\lambda>R}\frac{\de\lambda}{\lambda^{\eta}}\right)=O\!\left(\frac{(\log N)^{|\alpha|-\Re\alpha}}{R^{\eta-1}}\right).\nonumber
\eey
Summarizing, 
\bey
\int_{\R}\varphi_{\Omega,N}^{(k,\alpha)}(\lambda)\hat f(\lambda)\de\lambda&=&\int_{|\lambda|\leq R}\varphi^{(\alpha)}(\lambda)\hat f(\lambda)\de\lambda +O\!\left(\frac{\log\log N}{\log N}\right)+O\!\left(\frac{(\log N)^{|\alpha|-\Re\alpha}}{R^{\eta-1}}\right)+\label{summarizing-int-varphiNhatf0}\\
&&+\begin{cases}O\!\left(\frac{\log R}{\log N}\log\!\left(\frac{\log^2 N}{R}\right)\right),&\Re\alpha+\eta=2;\\
O\!\left(\frac{R^{2-\Re\alpha-\eta}}{\log N}\log\!\left(\frac{\log N}{R}\right)\right),&\mbox{otherwise}.\end{cases}\label{summarizing-int-varphiNhatf}
\eey

One can check that the term in (\ref{summarizing-int-varphiNhatf}) is dominated by those in (\ref{summarizing-int-varphiNhatf0}).
\end{proof}

Theorem \ref{thm-1} follows now immediately. In fact,
\be
S_{\Omega,f}(k,\alpha;N)=Z_{\Omega,N}^{(k,\alpha)}\int_{\R}\varphi_{\Omega,N}^{(k,\alpha)}(\lambda)\hat f(\lambda)\de\lambda,\nonumber
\ee
and Lemma \ref{lem-asymptotic-ZNkq} and Proposition \ref{prop-varphi-tau} give the desired statements
(\ref{statement-thm-1})-
(\ref{epsilon-in-main-theorem}).

\section{Proof of Theorem \ref{pw-convergence-varphi_xi_N}}\label{section-proof-of-theorem-pw...}
We shall need the following result

\begin{lem}\label{mathcal-E^k-alpha_Omega(u,lambda,N)}
Let \bey
&&\mathcal E^{(k,\alpha)}_\Omega(u,\lambda,N)=\frac{\de}{\de u}F_{\Omega,0}^{(k,\alpha)}(u)-\frac{\alpha}{u^2}+\left(\frac{\de}{\de u}F_{\Omega,1}^{(k,\alpha)}(u)+\frac{\alpha}{u^2}\right)e^{i\lambda\frac{\log u}{\log N}}+\sum_{t=2}^{k-1}\frac{\de }{\de u}F_{\Omega,t}^{(k,\alpha)}(u)e^{i\lambda t\frac{\log u}{\log  N}},\nonumber
\eey
where $|u|>|\alpha|$. Then
\be
\mathcal E^{(k,\alpha)}_\Omega(u,\lambda,N)=O\!\left(\sum_{j=1}^{k-1}\frac{\alpha^j\left(e^{i\lambda j\frac{\log u}{\log N}}-1\right)}{u^{j+2}}\right),\nonumber
\ee
where the constant implied by the $O$-notation depends on $j$, $k$ and $\alpha$.
\end{lem}

\begin{proof}
We have
\bey
&&\mathcal E^{(k,\alpha)}_\Omega(u,\lambda,N)=\nonumber\\
&&=\sum_{l=2}^\infty\frac{-l\,b_{\Omega,0}^{(k,\alpha)}(l)}{u^{l+1}}
+e^{i\lambda\frac{\log u}{\log N}}\sum_{l=2}^{\infty}\frac{-l\alpha\, b_{\Omega,0}^{(k,\alpha)}(l-1)}{u^{l+1}}
+\sum_{t=2}^{k-1}e^{i\lambda t\frac{\log u}{\log N}}\sum_{l=t+1}^\infty\frac{-l\alpha^t\, b_{\Omega,0}^{(k,\alpha)}(l-t)}{u^{l+1}}=\nonumber\\
&&=\sum_{l=2}^{\infty}\frac{-l}{u^{l+1}}\sum_{j=0}^{l\wedge (k-1)}\alpha^j b_{\Omega,0}^{(k,\alpha)}(l-j) e^{i\lambda j\frac{\log u}{\log N}}=\nonumber\\
&&=\sum_{l=2}^\infty\frac{-l}{u^{l+1}}\sum_{j=0}^{l\wedge (k-1)}\alpha^j b_{\Omega,0}^{(k,\alpha)}(l-j)\sum_{s=0}^\infty \frac{\left(i\lambda j \frac{\log u}{\log N}\right)^s}{s!}\label{Epsilon-before-removing-s=0},
\eey
where $l\wedge (k-1)$ denotes the minimum of $l$ and $k-1$.
The terms corresponding to $s=0$ in the sum (\ref{Epsilon-before-removing-s=0}) can be removed, as Lemmata \ref{structure-Laurent-coeffcients-b_Omega-1}-\ref{structure-Laurent-coeffcients-b_Omega-2} yield
\bey
\sum_{j=0}^{l\wedge(k-1)} \alpha^j b_{\Omega,0}(l-j)=0.\nonumber
\eey
In fact, if $2\leq l\leq k-1$, then $\sum_{j=0}^l\alpha^j b_{\Omega,0}^{(k,\alpha)}(l-j)=\alpha^{l-1}b_{\Omega,0}^{(k,\alpha)}(1)+\alpha^{l}b_{\Omega,0}^{(k,\alpha)}(0)=0$. If $l\geq k$, say $l=c k+d$, with $c\geq 1$ and $0\leq d\leq k-1$, then 
\be\sum_{j=0}^{k-1}\alpha^j b_{\Omega,0}^{(k,\alpha)}(l-j)=\begin{cases} b_{\Omega,0}^{(k,\alpha)}(c k)+\alpha^{k-1} b_{\Omega,0}^{(k,\alpha)}((c-1)k+1)=0,&\mbox{if $d=0$;}\\ \alpha^{d-1} b_{\Omega,0}^{(k,\alpha)}(ck+1)+
\alpha^{d}b_{\Omega,0}^{(k,\alpha)}(ck)=0,&\mbox{if $d\geq 1$}.\end{cases}\nonumber\ee
We proceed from (\ref{Epsilon-before-removing-s=0}), after noticing that the terms corresponding to $j=0$ can be removed too:
\bey
\mathcal E^{(k,\alpha)}_\Omega(u,\lambda,N)&=&\sum_{l=2}^\infty\frac{-l}{u^{l+1}}\sum_{j=0}^{l\wedge (k-1)}\alpha^j b_{\Omega,0}^{(k,\alpha)}(l-j)\sum_{s=1}^\infty \frac{\left(i\lambda j \frac{\log u}{\log N}\right)^s}{s!}=\nonumber\\
&=&\sum_{l=2}^\infty\frac{-l}{u^{l+1}}\sum_{j=1}^{l\wedge (k-1)}\alpha^j b_{\Omega,0}^{(k,\alpha)}(l-j)\left(e^{i\lambda j\frac{\log u}{\log N}}-1\right)=\nonumber\\
&=&\sum_{j=1}^{k-1}\alpha^j\left(e^{i\lambda j\frac{\log u}{\log N}}-1\right)\sum_{l=j+1}^\infty\frac{-l\, b_{\Omega,0}^{(k,\alpha)}(l-j)}{u^{l+1}}=\nonumber\\
&=&\sum_{j=1}^{k-1}\frac{\alpha^j\left(e^{i\lambda j\frac{\log u}{\log N}}-1\right)}{u^{j+2}}O(1), 
\nonumber
\eey
where the constant implied by the $O$-notation depends on $j$, $k$ and $\alpha$.

\end{proof}

\subsection{The main step}
We can write
\bey
\varphi_{\Omega,N}^{(k,\alpha)}(\lambda)&=&\sum_{x\in\frak X_N^{(k)}}\exp\!\left(i\lambda\sum_{p\leq N}
\nu(p)\frac{\log p}{\log N}\right) P_{\Omega,N}^{(k,\alpha)}(\{x\})
=\nonumber\\
&=&\sum_{\epsilon(p_1),\ldots,\epsilon(p_{\pi(N)})
\in\{0,1,\ldots,k-1\}
}\prod_{p\leq N}\exp\!\left(i\lambda \epsilon(p)\frac{\log p}{\log N}\right)P_{\Omega,N}^{(k,\alpha)}(\{\nu(p)=\epsilon(p)\})=\nonumber\\
&=&\prod_{p\leq N}  \sum_{t=0}^{k-1} e^{i\lambda t\frac{\log p}{\log N}}P_{\Omega,N}^{(k,\alpha)}(\{\nu(p)=t\})
.\nonumber
\eey
Now, by Lemma \ref{lem-distri-nu_j} and (\ref{def-F_Omega,t}
), we get
\bey
\varphi_{\Omega,N}^{(k,\alpha)}(\lambda)&=&\prod_{p\leq N}\left(F_{\Omega,0}^{(k,\alpha)}(p)+F_{\Omega,1}^{(k,\alpha)}(p)e^{i\lambda\frac{\log p}{\log N}}+\sum_{t=2}^{k-1}e^{i\lambda t\frac{\log p}{\log N}}F_{\Omega,t}^{(k,\alpha)}(p)\right)=\nonumber\\
&=&\prod_{p\leq N}\left(1+\frac{\alpha}{p}\left(e^{i\lambda\frac{\log p}{\log N}}-1\right)+\sum_{t=2}^{k-1} G_{\Omega,t,N}^{(k,\alpha)}(p,\lambda)\right),\label{pf-main-thm-1}
\eey
where
\bey
G_{\Omega,t,N}^{(k,\alpha)}(u,\lambda)=\begin{cases}
F_{\Omega,0}^{(k,\alpha)}(u)-1+\frac{\alpha}{u}+\left(F_{\Omega,1}^{(k,\alpha)}(u)-\frac{\alpha}{u}\right)e^{i\lambda\frac{\log u}{\log N}}+F_{\Omega,2}^{(k,\alpha)}(u)e^{2i\lambda\frac{\log u}{\log N}},& t=2;\\
F_{\Omega,t}^{(k,\alpha)}(u)e^{i\lambda t\frac{\log u}{\log N}},&3\leq t\leq k-1.
\end{cases}\nonumber
\eey
Lemmata \ref{structure-Laurent-coeffcients-b_Omega-1} and \ref{structure-Laurent-coeffcients-b_Omega-2} 
imply 
that 
$G_{\Omega,t,N}^{(k,\alpha)}(u,\lambda)=O(1/u^t) 
$, for $2
\leq t\leq k-1$, where the constants implied by the $O$-notations depend only on 
$k$ and $\alpha$.
This means, in particular, that there exists $d_*=d_*(k,\alpha)$ such that for every $\lambda\in\R$, every $N\geq 2$ and every $p>d_*$, the complex numbers $$z(p)=z^{(k,\alpha)}(p,\lambda,N)=1+\frac{\alpha}{p}\left(e^{i\lambda\frac{\log p}{\log N}}-1\right)+\sum_{t=2}^{k-1}G_{\Omega,t,N}^{(k,\alpha)}(p,\lambda)$$ belong to the open disk of radius $1/2$, centered at $1$. Thus we can write $\log z(p)=\log|z(p)|+i\arg z(p)$ and choose a branch of the logarithm consistently for all $p>d_*$. Let us also assume that $d_*>|\alpha|$ (we shall use this fact later). We get 
\be
\varphi_{\Omega,N}^{(k,\alpha)}(\lambda)=\prod_{p\leq d_{*}}z^{(k,\alpha)}(p,\lambda,N)\cdot \tilde\varphi_{\Omega,N}^{(k,\alpha)}(\lambda) \label{varphi=prod-cdot-tilde-varphi}
\ee
where, after taking the logarithm,
\bey
\log\tilde\varphi_{\Omega,N}^{(k,\alpha)}(\lambda)&=&\sum_{d_*<p\leq N}\log z^{(k,\alpha)}(p,\lambda,N)\nonumber.
\eey
Abel summation yields
\bey
\log \tilde\varphi_{\Omega,N}^{(k,\alpha)}(\lambda)=\pi(N)\mathcal L(N,\lambda,N)-\pi(d_*)\mathcal L(d_*,\lambda,N)-\int_{d_*}^N\pi(u)\frac{\partial }{\partial u}\mathcal L(u,\lambda,N)\de u,\label{pf-main-thm-Abel}
\eey
where $\mathcal L(u,\lambda,N)=\mathcal L^{(k,\alpha)}(u,\lambda,N)=\log z^{(k,\alpha)}(u,\lambda, N)$. 

\begin{remark}
As pointed out by A.J. Hildebrand to the author, one can try to estimate (\ref{pf-main-thm-1}) using a result by Tenenbaum (\cite{Tenenbaum-book}, Chapter III.5). This approach will be the subject of future investigation. Likely, it will allow for a wider range for $R$ in Theorem \ref{thm-1} and reduce the error term in the case of small $\eta$.
\end{remark}

\subsection{The boundary terms}
Let us estimate the first two terms in the rhs of (\ref{pf-main-thm-Abel}). By Lemmata \ref{structure-Laurent-coeffcients-b_Omega-1}-\ref{structure-Laurent-coeffcients-b_Omega-2} 

\bey
\pi(N)\mathcal L(N,\lambda,N)
&=& \pi(N)
\log\!\left(1+\frac{\alpha}{N}\left(e^{i\lambda}-1\right)+\sum_{t=2}^{k-1}G_{\Omega,t,N}^{(k,\alpha)}(N,\lambda)\right)=\nonumber\\
&=&O\!\left(\frac{N}{\log N}\right)\log\!\left(1+O\!\left(\frac{1}{N}\right)(e^{i\lambda}-1)+O\!\left(\frac{1}{N^2}\right)(1+e^{i\lambda}+e^{2i\lambda})+\right.\nonumber\\
&&\left.+\sum_{t=3}^{k-1}O\!\left(\frac{1}{N^t}\right)e^{i t \lambda}\right)=O\!\left(\frac{1}{\log N}\right),\label{O-estimate-boundary-term-1}\\
\pi(d_*)\mathcal L(d_*,\lambda,N)&=&\pi(d_*)\log\!\left(1+\frac{\alpha}{d_*}\left(e^{i\lambda\frac{d_*}{\log N}}-1\right)+O\!\left(\frac{1}{N^2}\right)\right)= O\!\left(\varepsilon\right)+ O\!\left(\tfrac{1}{N^2}\right) 
\label{O-estimate-boundary-term-2}
,\eey
where $d_*$ is as above, and  the constants implied by the $O$-notation depend only on 
$k$ and $\alpha$. The two boundary terms in the rhs of (\ref{pf-main-thm-Abel}) are therefore $O\!\left(\tfrac{1}{\log N}\right)+O(\varepsilon)$ 
as $N\to\infty$.

\subsection{The integral term}
Let us now estimate the integral in the rhs of (\ref{pf-main-thm-Abel}).
Recall the functions $G_{\Omega,t,N}^{(k,\alpha)}(u,\lambda)$ and $\mathcal E^{(k,\alpha)}_{\Omega}(u,\lambda,N)$ introduced above. We have
\bey
&&\frac{\partial}{\partial u}\mathcal L(u,\lambda,N)=\frac{1}{z^{(k,\alpha)}(u,\lambda,N)}\left(-\frac{\alpha}{u^2}\left(e^{i\lambda\frac{\log u}{\log N}}-1\right)+\frac{i\alpha \lambda
}{u^2 \log N}e^{i\lambda\frac{\log u}{\log N}}+\frac{\de}{\de u}F_{\Omega,0}^{(k,\alpha)}(u)-\frac{\alpha}{u^2}+\right.\nonumber\\
&&\left.+\left(\frac{\de}{\de u}F_{\Omega,1}^{(k,\alpha)}(u)+\frac{\alpha}{u^2}\right)e^{i\lambda\frac{\log u}{\log N}}+\frac{i\lambda
}{u \log N}\left(F_{\Omega,1}^{(k,\alpha)}(u)-\frac{\alpha}{u}\right)e^{i\lambda\frac{\log u}{\log N}}+\right.\nonumber\\
&&\left.+\sum_{t=2}^{k-1}\left(\frac{\de }{\de u}F_{\Omega,t}^{(k,\alpha)}(u)e^{i\lambda t\frac{\log u}{\log  N}}+ \frac{i\lambda\, t
}{u\log N}F_{\Omega,t}^{(k,\alpha)}(u)e^{i\lambda t\frac{\log u}{\log N}}\right)\right)=\nonumber\\
&&=\left(1-\frac{\frac{\alpha}{u}\left(e^{i\lambda\frac{\log u}{\log N}}-1\right)+\sum_{t=2}^{k-1}G_{\Omega,t,N}^{(k,\alpha)}(u,\lambda)}{z^{(k,\alpha)}(u,\lambda,N)}\right)\left(-\frac{\alpha}{u^2}\left(e^{i\lambda\frac{\log u}{\log N}}-1\right)+ \frac{i\alpha \lambda}{u^2\log N}e^{i\lambda\frac{\log u}{\log N}}+\right.\nonumber\\
&&\left.+\mathcal E_\Omega^{(k,\alpha)}(u,\lambda,N)+
\frac{\lambda 
}{u^3\log N}O(1)+
\sum_{t=2}^{k-1}
\frac{\lambda 
}{u^{t+1}\log N}O(1)\right)
\label{derivative-L-0}\eey

 Since $d_*\leq u\leq N$, let us notice that $\frac{\alpha}{u}\left(e^{i\lambda\frac{\log u}{\log N}}-1\right)\frac{1}{z(u,\lambda,N)}=
 \frac{1}{u}\left(e^{i\lambda\frac{\log u}{\log N}}-1\right)O(1)$. 
Moreover, $\frac{1}{z(u,\lambda, N)}\sum_{t=2}^{k-1}G_{\Omega,t,N}^{(k,\alpha)}(u,\lambda)=O\!\left(\frac{1}{u^2}\right)$. Thus, the first bracket in (\ref{derivative-L-0}) can be written as $I_{1,1}+I_{1,2}+I_{1,3}$, where $I_{1,j}=I_{\Omega,1,j}^{(k,\alpha)}(u,\lambda,N)$, $j=1,2,3$, and  
 \bey
 I_{1,1}&=&1,\nonumber\\
I_{1,2}&=&
\frac{O(1)}{u}\left(e^{i\lambda\frac{\log u}{\log N}}-1\right),\nonumber\\
I_{1,3}&=& 
O\!\left(\frac{1}{u^2}\right).\nonumber
 \eey
 Let us look at the second bracket in (\ref{derivative-L-0}). 
By Lemma (\ref{mathcal-E^k-alpha_Omega(u,lambda,N)})
 we get $k-1$ terms of the form $\frac{O(1)}{u^{j+2}}\left(e^{i\lambda j \frac{\log u}{\log N}}-1\right)$, $1\leq j\leq k-1$ (it would not be enough to estimate them as $O(1/u^{j+2})$ as we want a better control of error terms).  The implied constants can be chosen in order to not depend upon $j$ but only on $k$ and $\alpha$. Let us also combine the last two terms in (\ref{derivative-L-0}) into $O\!\left(\frac{\varepsilon}{u^3}\right)$. 
This means that the the second bracket in (\ref{derivative-L-0}) can be written as $I_{2,1}+I_{2,2}+\ldots+I_{2,k+2}$, where $I_{2,j}=I_{\Omega,2,j}^{(k,\alpha)}(u,\lambda,N)$, $1\leq j\leq k+2$, and
\bey
I_{2,1}&=&-\frac{\alpha}{u^2}\left(e^{i\lambda\frac{\log u}{\log N}}-1\right),\nonumber\\ 
I_{2,2}&=&\frac{i\alpha\varepsilon}{u^2}e^{i\lambda\frac{\log u}{\log N}},\nonumber\\ 
I_{2,j}&=&\frac{O(1)\left(e^{i\lambda(j-2)\frac{log u}{\log N}}-1\right)}{u^{j}},\hspace{.4cm}3\leq j\leq k+1,\nonumber\\ 
I_{2,k+2}&=&
O\!\left(\frac{\varepsilon}{u^3}\right).\nonumber 
\eey

Recall that all the constants implied by the above $O$-notations 
depend only on $
k,\alpha$, and not on $\lambda$, $u$, and $N$. 
Let us also write \be\pi(u)=\frac{u}{\log u}+O\!\left(\frac{u}{\log^2 u}\right)=I_{0,1}+I_{0,2}.\nonumber\ee 
The integral in the rhs of (\ref{pf-main-thm-Abel}) becomes now 

\bey
-\int_{d_*}^N \pi(u)\frac{\partial}{\partial u}\mathcal L(u,\lambda, N)\de u&=&-\sum_{j=1}^2\sum_{j'=1}^3\sum_{j''=1}^{k+2} \int_{d_*}^N I_{0,j}I_{1,j'}I_{2,j''}\de u\label{12integrals}.
\eey

We claim that, amongst the $6k+12$ integrals in (\ref{12integrals}), the one corresponding to $j=j'=j''=1$ is the main term, 
and the remaining $6k+11$ integrals are $O\!\left(\frac{1}{\log N}\right)+O(\varepsilon\log|\varepsilon|)$. 
Let us perform the change of variables $v=\frac{\log u}{\log N}$. We have
\bey
J_{1,1,1}&=&-
\int_{d_*}^N I_{0,1}I_{1,1}I_{2,1}\de u 
=\alpha \int_{d_*}^N\frac{e^{i\lambda \frac{\log u}{\log N}}-1}{u\log u}\de u=\alpha\int_{\log d_*/\log N}^{1}\frac{e^{i\lambda v}-1}{v}\de v=\nonumber\\
&=&
\alpha \int_0^1\frac{e^{i\lambda v}-1}{v}\de v+O(\varepsilon). 
\label{J111} 
\eey

The fact that 
\be J_{j,j',j''}=-\int_{d_*}^N I_{0,j}I_{1,j'}I_{2,j''}\de u=O\!\left(\frac{1}{\log N}\right)+O(\varepsilon\log|\varepsilon|)\label{J-allothers}\ee for all other values of $j,j',j''$ is shown in Appendix A.
By (\ref{pf-main-thm-Abel}-
\ref{O-estimate-boundary-term-2}, \ref{12integrals}-\ref{J-allothers}
), we have that

\be\label{estimate-log-tilde-varphi-two-cases}
\log\tilde\varphi_{*,N}^{(k,\alpha)}(\lambda)=
\alpha\int_{0}^1\frac{e^{i\lambda v}-1}{v}\de v+O\!\left(\frac{1}{\log N}\right)+O\!\left(\varepsilon\log|\varepsilon|\right). 
\ee

The first factor in the rhs of (\ref{varphi=prod-cdot-tilde-varphi}) can be written as follows as $N\to\infty$
\bey
\prod_{p\leq d_*}z_{}^{(k,\alpha)}(p,\lambda,N)&=&\prod_{p\leq d_*}\sum_{t=0}^{k-1}e^{i\lambda\frac{\log p}{\log N}}F_{\Omega,t}^{(k,\alpha)}(p)\nonumber\\ 
&=&
\prod_{p\leq d_*}\left(\sum_{t=0}^{k-1} F_{\Omega,t}^{(k,\alpha)}(p)+O(\varepsilon)\right)=1+O(\varepsilon),\label{varphi=prod-cdot-tilde-varphi-1}
\eey
where we used the fact that
\bey
\sum_{t=0}^{k-1}F_{\Omega,t}^{(k,\alpha)}(p)=\frac{\sum_{t=0}^{k-1}\alpha^t p^{k-1-t}}{p^{k-1}+\alpha p^{k-2}+\alpha^2 p^{k-3}+\ldots+\alpha^{k-1}}=1.\nonumber
\eey
Now, combining (\ref{estimate-log-tilde-varphi-two-cases}) and (\ref{varphi=prod-cdot-tilde-varphi-1}), we get 
\bey
\varphi_{\Omega,N}^{(k,\alpha)}(\lambda)&=&\left(1+O(\varepsilon)\right)\exp\!\left(\alpha\int_0^1\frac{e^{i\lambda v}-1}{v}\de v +O\!\left(\frac{1}{\log N}\right)+O\!\left(\varepsilon\log|\varepsilon|\right)\right)=\nonumber\\
&=&\exp\!\left(\alpha\int_0^1\frac{e^{i\lambda v}-1}{v}\de v\right)\left(1+O\!\left(\frac{1}{\log N}\right)+O\!\left(\varepsilon\log|\varepsilon|\right)\right)\nonumber
\eey
and this concludes the proof of Theorem \ref{pw-convergence-varphi_xi_N}.

\section{An Example}\label{section-example}
We pointed out that, for general $\alpha$ and $f$, the integral 
\be I=\int_{|\lambda|\leq R}\hat f(\lambda)\varphi^{(\alpha)}(\lambda)\de\lambda\nonumber\ee 
in (\ref{statement-thm-1}) may tend to zero as $N\to\infty$ (recall: $R=R(N)$). In this section we discuss an example where $\alpha=-1$ and the above integral is bounded away from zero as $N\to\infty$. We have the following
\begin{prop}
Let $f(u)=\bm 1_{[-1,1]}e^{-\frac{1}{1-u^2}}$. Then, for sufficiently large $N$, 
\be
\left|\,\int_{|\lambda|\leq R}\hat f(\lambda)\varphi^{(-1)}(\lambda)\de\lambda\right|\geq \frac{3}{100}.\label{statement-prop-example}
\ee
\end{prop}
\begin{proof}
Since $S_{\Omega, f}(k,-1;N)$ is real, then by Theorem \ref{thm-1}, $\Im I 
=o(1)$ as $N\to\infty$. Thus, for every $\delta>0$, there exists $N_\delta$ such that $|\Im I|\leq \delta$ for every $N\geq N_\delta$. Let us now focus on  $\Re I$. Notice that $\hat f$ is real-valued and thus $\Re I=\int_{|\lambda|\leq R}\hat f(\lambda)\Re \varphi^{(-1)}(\lambda)\de\lambda$. 
We can write
\be\Re\varphi^{(-1)}(\lambda)=e^{\gamma-\mathrm{Ci}(\lambda)}\lambda\cos\left(\mathrm{Si}(\lambda)\right)\label{example-Re-varphi-1(lambda)}
\ee
where $\mathrm{Ci(\lambda)}=-\int_\lambda^\infty \frac{\cos t\de t}{t}$ and $\mathrm{Si}(\lambda)=\int_0^\lambda\frac{\sin t\de t}{t}$. 
We can use the expression (\ref{example-Re-varphi-1(lambda)}) to obtain the estimate\footnote{Although $|\varphi^{(-1)}(\lambda)|=O(\lambda)$ as $\lambda\to\infty$, only the imaginary part of the function $\varphi^{(-1)}$ is unbounded. This property holds true only for $\alpha=-1$.} $|\Re\varphi^{(-1)}(\lambda)|\leq e^\gamma$, valid for all $\lambda\in\R$. 
Let us write 
\be \Re I=\int_{|\lambda|\leq R}\hat f(\lambda)\Re\varphi^{(-1)}(\lambda)\de\lambda=\int_{|\lambda|\leq r}+\int_{r<|\lambda|\leq R}=I_1+I_2\label{two-integrals-example}\ee
where $r>0$ does not depend on $N$ and will be chosen later.
The idea is that $I_1$ can be estimated numerically with arbitrary precision using, for example, a quadrature method for the integral. More precisely, let $F(\lambda)=\hat f(\lambda)\Re\varphi^{(-1)}(\lambda)$ and observe that $F(\lambda)=F(-\lambda)$, so that $\int_{|\lambda|\leq r}F(\lambda)\de\lambda=2\int_0^rF(\lambda)\de\lambda$. We have

\be \int_0^rF(\lambda)\de\lambda=h\sum_{m=1}^MF\!\left(h(m-\tha)\right)+\mathcal E(r,M),\nonumber\ee
where $h=\frac{r}{M}$ and $\mathcal E(r,M)=\frac{r h^2}{24}F''(\rho)$ for some $0<\rho<r$. The graphs of $F$, $F'$ and $F''$ are shown in Figure \ref{fig-FF'F''}. One can see that $|F''(\lambda)|\leq \tha$.
\begin{figure}[ht!]
\begin{center}
\includegraphics[width=17cm]{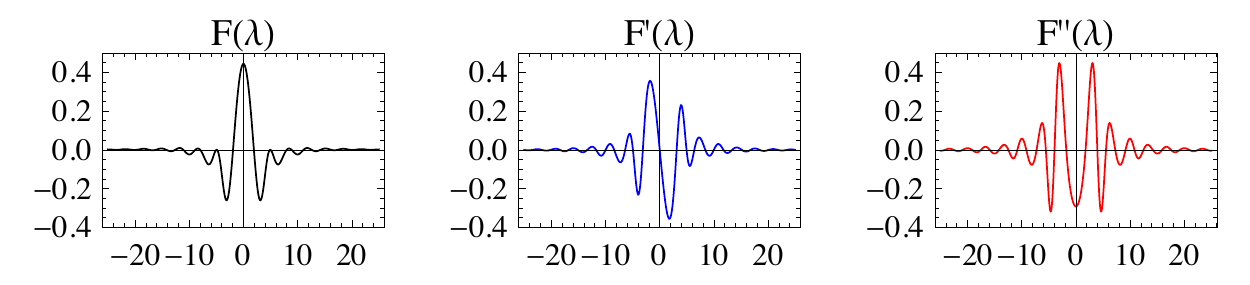}
\caption{The graphs of $F$, $F'$ and $F''$.}\label{fig-FF'F''}
\end{center}
\end{figure}
Since $F$ can be estimated to an arbitrary precision, we can estimate the Riemann sum $h\sum_{m=1}^MF\!\left(h(m-\tha)\right)$ arbitrarily well for fixed $r$ and $M$, and we have the estimate $|\mathcal E(r,M)|\leq \frac{h^2 r}{48}$.
The integral $I_2$ in (\ref{two-integrals-example}) can be estimated explicitly as follows:
\be \left|\,\int_{r<|\lambda|\leq R}F(\lambda)\de\lambda\right|\leq 2\int_{r}^R|F(\lambda)|\de\lambda.\nonumber\ee
By means of stationary phase method, it can be shown that the function $\hat f$ satisfies the asymptotic
\be
\hat f(\lambda)\sim\frac{1}{\pi}\,\Re\left\{\sqrt{\frac{-i\pi}{\sqrt{2i}\lambda^{\frac{3}{2}}}}e^{i\lambda-\frac{1}{4}-\sqrt{2i\lambda}}\right\}\nonumber
\ee 
as $\lambda\to\infty$. 
Moreover, one can check that $|\hat f(\lambda)|\leq \frac{3}{2\pi}e^{-\sqrt \lambda}\lambda^{-3/4}$ for $\lambda\geq 1$. This implies that $|F(\lambda)|\leq \frac{3 e^\gamma}{2\pi}e^{-\sqrt \lambda}\lambda^{-3/4}$. We get
\be \left|\,\int_{r<|\lambda|\leq R}F(\lambda)\de\lambda\right|\leq \frac{3e^\gamma}{\pi }\int_{r}^Re^{-\sqrt \lambda}\lambda^{-3/4}\de\lambda
=\frac{6 e^\gamma}{\sqrt\pi}\left[\mathrm{erf}(\lambda^{1/4})\right]_{\lambda=r}^{\lambda=R}\leq \frac{6 e^\gamma}{\sqrt\pi}\,\mathrm{erfc}(r^{1/4}),\nonumber\ee
where $\mathrm{erf}(x)=\frac{2}{\sqrt \pi}\int_0^x e^{-t^2}\de t$ and $\mathrm{erfc}(x)=1-\mathrm{erf}(x)$.
Now, let $r=5$ and $M=1000$. We have
\bey
2h\sum_{m=1}^MF\!\left(h(m-\tha)\right)&=&0.23821680383626264857\pm10^{-20},\nonumber\\
\frac{h^2 r}{24}&=&5.208(3)\cdot 10^{-6},\nonumber\\
\frac{6e^\gamma}{\sqrt\pi}\mathrm{erfc}(r^{1/4})&=&0.20771652138513808389\pm10^{-20},\nonumber
\eey
and therefore
\bey\left|\int_{|\lambda|\leq R}F(\lambda)\de\lambda\right|&\geq& \left|2h\sum_{m=1}^MF\!\left(h(m-\tha)\right)\right|-\left|\mathcal E(r,M)\right|-\left|\,\int_{r<|\lambda|\leq R}F(\lambda)\de\lambda\right|\geq\frac{3}{100}\nonumber
\eey
This shows that, for sufficiently large $N$,
\be\left|\int_{|\lambda|\leq R(N)}\hat f(\lambda)\varphi^{(-1)}(\lambda)\de\lambda\right|>\frac{3}{100}+\frac{1}{2500}.\label{example-final-estimate}\ee
Now let us choose $\delta =\frac{1}{2500}$ so that $|\Im I|\leq \delta$ for sufficiently large $N$. This, together with (\ref{example-final-estimate}),  yields the desired statement (\ref{statement-prop-example}).
\end{proof}

\appendix
\section{Estimate of the error terms}

This appendix contains the estimates of the integrals $J_{j,j',j''}$ from (\ref{J-allothers}) for $1\leq j\leq 2$, $1\leq j'\leq 3$ and $1\leq j''\leq k+2$, except for $(j,j',j'')=(1,1,1)$ already 
treated in (\ref{J111}). Recall that $\varepsilon=\frac{\lambda}{\log N}$ is assumed to tend to zero as $N\to\infty$. We will assume for simplicity that $\lambda>0$.
We group error terms in different sections, according to the methods used to estimate them.
\subsection{$O(\varepsilon)$ terms} These error terms are very easy, due to the fact that $I_{2,2}$ and $I_{2,k+2}$ are $O(\varepsilon/u^2)$  and $O(\varepsilon/u^3)$ respectively and the corresponding error terms can be written as $O(\varepsilon \mathcal I)$, where $\mathcal I$ is an absolutely convergent integral. We have

\bey
J_{1,2,2}&=&-\int_{d_*}^N I_{0,1}I_{1,2}I_{2,2}\de u=O\!\left(\varepsilon\int_{d_*}^N\frac{\left(e^{i\lambda\frac{\log u}{\log N}}-1\right)e^{i\lambda\frac{\log u}{\log N}}}{u^2\log u}\de u\right)=O(\varepsilon),\nonumber\\
J_{1,3,2}&=&-\int_{d_*}^N I_{0,1}I_{1,3}I_{2,2}\de u=O\!\left(\varepsilon\int_{d_*}^N\frac{e^{i\lambda\frac{\log u}{\log N}}}{u^3\log u}\de u\right)=O(\varepsilon),\nonumber\\
J_{2,1,2}&=&-\int_{d_*}^N I_{0,2}I_{1,1}I_{2,2}\de u=O\!\left(\varepsilon\int_{d_*}^N\frac{e^{i\lambda\frac{\log u}{\log N}}}{u \log^2 u}\de u\right)=O(\varepsilon),\nonumber\\
J_{2,2,2}&=&-\int_{d_*}^N I_{0,2}I_{1,2}I_{2,2}\de u=O\!\left(\varepsilon\int_{d_*}^N\frac{\left(e^{i\lambda\frac{\log u}{\log N}}-1\right)e^{i\lambda\frac{\log u}{\log N}}}{u^2\log^2 u}\de u\right)=O(\varepsilon),\nonumber\\
J_{2,3,2}&=&-\int_{d_*}^N I_{0,2}I_{1,3}I_{2,2}\de u=O\!\left(\varepsilon\int_{d_*}^N\frac{e^{i\lambda\frac{\log u}{\log N}}}{u^3 \log^2 u}\de u\right)=O(\varepsilon),\nonumber\\
J_{1,1,k+2}&=&-\int_{d_*}^N I_{0,1}I_{1,1}I_{2,k+2}\de u=O\!\left(\varepsilon \int_{d_*}^N\frac{\de u}{u^2\log u}\right)=O(\varepsilon),\nonumber\\
J_{1,2,k+2}&=&-\int_{d_*}^N I_{0,1}I_{1,2}I_{2,k+2}\de u=O\!\left(\varepsilon\int_{d_*}^N\frac{\left(e^{i\lambda\frac{\log u}{\log N}}-1\right)}{u^3\log u}\de u\right)=O(\varepsilon),\nonumber\\
J_{1,3,k+2}&=&-\int_{d_*}^N I_{0,1}I_{1,3}I_{2,k+2}\de u=O\!\left(\varepsilon\int_{d_*}^N\frac{\de u}{u^4\log u}\right)=O(\varepsilon),\nonumber\\
J_{2,1,k+2}&=&-\int_{d_*}^N I_{0,2}I_{1,1}I_{2,k+2}\de u=O\!\left(\varepsilon \int_{d_*}^N \frac{\de u}{u^2\log^2 u}\right)=O(\varepsilon),\nonumber\\
J_{2,2,k+2}&=&-\int_{d_*}^N I_{0,2}I_{1,2}I_{2,k+2}\de u=O\!\left(\varepsilon \int_{d_*}^N \frac{\left(e^{i\lambda\frac{\log u}{\log N}}-1\right)}{u^3\log^2 u}\de u\right)=O(\varepsilon),\nonumber\\
J_{2,3,k+2}&=&-\int_{d_*}^N I_{0,2}I_{1,3}I_{2,k+2}\de u=O\!\left(\varepsilon \int_{d_*}^N \frac{\de u}{u^4\log^2 u}\right)=O(\varepsilon).\nonumber
\eey

\subsection{$O(\varepsilon\log\varepsilon)$ terms}
The analysis of these error terms yields estimates of the form $O\!\left(\frac{1}{\log N}\right)+O(\varepsilon\log\varepsilon)$.
We shall need the special function 
\bey
\mathrm{Ei}(z)=-\int_{-z}^\infty \frac{e^{-t}}{t}\de t, \hspace{.3cm}z\in\R,\nonumber
\eey
where the integral is in the sense principal value due to the singularity at zero. For complex arguments $\mathrm{Ei}(z)$ is defined by analytic continuation. Notice that $v\mapsto\mathrm{Ei}(iv)$ is the antiderivative of $v\mapsto e^{iv}/v$. Moreover, for $\tau,w\in\R$, 
\bey
\mathrm{Ei}(\tau)&=&\gamma+\log \tau+O(\tau)\hspace{.3cm}\mbox{as $\tau\to0+$};\label{property-Ei-0}\\
\mathrm{Ei}(i \tau)&=&\gamma+\frac{i\pi}{2}+\log \tau+O(\tau)\hspace{.3cm}\mbox{as $\tau\to0+$};\label{property-Ei-C0}\\
\mathrm{Ei}(w)&=& \frac{e^{w}}{w}\left(1+O\!\left( \frac{1}{w}\right)\right)\hspace{.3cm}\mbox{as $w\to\infty$};\label{property-Ei-infty}\\
\mathrm{Ei}(i w)&=& i\pi-\frac{e^{i w}}{w}\left(1+O\!\left( \frac{1}{w}\right)\right) \hspace{.3cm}\mbox{as $w\to\infty$};\label{property-Ei-Cinfty}
\eey
see \cite{Abramowitz-Stegun}.
By (\ref{property-Ei-C0}, \ref{property-Ei-Cinfty}) we have
\bey
J_{1,1,2}&=&-\int_{d_*}^N I_{0,1}I_{1,1}I_{2,2}\de u=-i\alpha\varepsilon\int_{d_*/\log N}^1\frac{e^{i\lambda v}}{v}\de v=-i\alpha\varepsilon\left.\mathrm{Ei}(i\lambda v)\right|_{v=d_*/\log N}^{v=1}=\nonumber\\
&=&-i\alpha\varepsilon\left(\mathrm{Ei}(i\lambda)-\mathrm{Ei}(i\varepsilon d_*)\right)=\nonumber\\
&=&O\!\left(\varepsilon\left(O(1)-\gamma-\frac{i\pi}{2}-\log \varepsilon+\log d_*+O(\varepsilon)\right)\right)=O(\varepsilon\log\varepsilon).\nonumber 
\eey
Let $\mathrm{Ci}$ and $\mathrm{Si}$ be the special functions introduced in Section \ref{section-example}. They satisfy the estimates
\bey
i\mathrm{Ci}(y)-\mathrm{Si}(y)&=&O(1)\hspace{.3cm}\mbox{for $|y|\geq c>0$};\label{Ci+Si-infty}\\
i\mathrm{Ci}(\tau)-\mathrm{Si}(\tau)&=&i(\gamma+\log \tau)-\tau+O(\tau^2)\hspace{.2cm}\mbox{as $\tau\to0$}\label{Ci+Si-0};
\eey
see \cite{Abramowitz-Stegun}.
We have
\bey
J_{2,1,1}&=&-\int_{d_*}^N I_{0,2}I_{1,1}I_{2,1}\de u=O\!\left(\int_{d_*}^N\frac{\left(e^{i\lambda\frac{\log u}{\log N}}-1\right)}{u\log^2 u}\de u\right)=O\!\left(\left.\tilde J_{2,1,1}(x)\right|_{\log d_*}^{\log N} 
\right),\nonumber
\eey
where
\bey
\tilde J_{2,1,1}(x)&=&\frac{1}{x}-\frac{e^{i\varepsilon x}}{x}+
\varepsilon\left(i \mathrm{Ci}\!\left(\varepsilon x\right)-\mathrm{Si}\!\left(\varepsilon x\right)\right).\nonumber
\eey
By (\ref{Ci+Si-infty}, \ref{Ci+Si-0}) we get
\bey
J_{2,1,1}&=&O\!\left(\frac{1}{\log N}\right)+
O(\varepsilon\log\varepsilon).\nonumber
\eey

\subsection{Mixed terms}
Here we present the estimates for error terms of order $O(\varepsilon)+O(\frac{1}{\log N})$. In this section we assume that $3\leq j\leq k+1$.  
The following integrals can be estimated using the properties (\ref{property-Ei-0}, \ref{property-Ei-infty}) of the exponential integral function $\mathrm{Ei}$.

\bey
J_{2,1,j}&=&-\int_{d_*}^N I_{0,2}I_{1,1}I_{2,j}\de u=O\!\left(\int_{d_*}^N\frac{\left(e^{i\lambda(j-2)\frac{\log u}{\log N}}-1\right)}{u^{j-1}\log^2u}\de u\right)=
O\!\left(\left.\tilde J_{2,1,j}(x)\right|_{\log d_*}^{\log N}
\right),\nonumber\\
J_{2,2,j}&=&-\int_{d_*}^N I_{0,2}I_{1,2}I_{2,j}\de u=O\!\left(\int_{d_*}^N\frac{\left(e^{i\lambda\frac{\log u}{\log N}}-1\right)\left(e^{i\lambda(j-2)\frac{\log u}{\log N}}-1\right)}{u^j\log^2u}\de u\right)=
O\!\left(\left.\tilde J_{2,2,j}(x)\right|_{\log d_*}^{\log N}
\right),\nonumber\\
J_{2,3,j}&=&-\int_{d_*}^N I_{0,2}I_{1,3}I_{2,j}\de u=O\!\left(\int_{d_*}^N\frac{\left(e^{i\lambda(j-2)\frac{\log u}{\log N}}-1\right)}{u^{j+1}\log^2 u}\de u\right)=
O\!\left(\left.\tilde J_{2,3,j}(x)\right|_{\log d_*}^{\log N} 
\right),\nonumber
\eey
where
\bey
\tilde J_{2,1,j}(x)&=&\frac{1}{x}\left(e^{-(j-2)x}-e^{-(j-2)\left(1-i\varepsilon\right)x}\right)+\nonumber\\
&&+
(j-2)\mathrm{Ei}\!\left(-(j-2)x\right)-(j-2)\!\left(1-i\varepsilon \right)\mathrm{Ei}\!\left(-(j-2)\!\left(1-i\varepsilon\right)x\right),\nonumber\\
\tilde J_{2,2,j}(x)&=&\frac{1}{x}\left(e^{-\left(j-1-i\varepsilon\right)x}-e^{-(j-1)\left(1-i\varepsilon \right)x}+e^{-\left(j-1-(j-2)i\varepsilon\right)x}-e^{-(j-1)x}\right)+\nonumber\\
&&+\left(j-1-i\varepsilon\right)\mathrm{Ei}\!\left(-\left(j-1-i\varepsilon\right)x\right)-\nonumber\\
&&-(j-1)\left(1-i\varepsilon\right)\mathrm{Ei}\!\left(-(j-1)\left(1-i\varepsilon\right)x\right)+\nonumber\\
&&+\left(j-1-(j-2)i\varepsilon\right)\mathrm{Ei}\!\left(-\left(j-1-(j-2)i\varepsilon\right)x\right)-\nonumber\\
&&-(j-1)\mathrm{Ei}(-(j-1)x),\nonumber
\\
\tilde J_{2,3,j}(x)&=&\frac{1}{x}\left(e^{-jx}-e^{-\left(j-(j-2)i\varepsilon\right)x}\right)+\nonumber\\
&&+j \mathrm{Ei}(-jx)-\left(j-(j-2)i\varepsilon\right)\mathrm{Ei}\!\left(-\left(j-(j-2)i\varepsilon\right)x\right).\nonumber
\eey
By (\ref{property-Ei-0}) and (\ref{property-Ei-infty}) we get
\bey
J_{2,1,j}&=&O\!\left(\frac{1}{N^{j-2}\log N}\right)+O\!\left(\varepsilon\right),\nonumber\\ 
J_{2,2,j}&=&O\!\left(\frac{1}{N^{j-1}\log N}\right)+O\!\left(\varepsilon^2\right),\nonumber\\
J_{2,3,j}&=&O\!\left(\frac{1}{N^j\log N}\right)+O\!\left(\varepsilon\right).\nonumber
\eey

For the remaining error terms we shall need the incomplete gamma function
\bey
\Gamma(a,z)&=&\int_z^\infty t^{a-1}e^{-t}\de t, \hspace{.3cm}z\in\R,
\eey
and defined for complex $z$ by analytic continuation. We will only need the cases $a=0$ and $a=-1$. It satisfies, for $w,\tau\in\R$ and $c,z\in\C$, 
\bey
\Gamma(0,w)&=&e^{-w}\left(\frac{1}{w}+O\!\left(\frac{1}{w^2}\right)\right),\hspace{.5cm}w\to\infty;\label{asymptotic-Gamma(0,w)}\\
\Gamma(a,c+z\tau )&=&\Gamma(a,c)-\frac{e^{-c}}{c^{1-a}} z\tau +O(\tau^2),\hspace{.5cm}\tau\to 0;\label{asymptotic-Gamma(0,c+z)}
\eey
see \cite{Abramowitz-Stegun}.
We have
\bey
J_{1,1,j}&=&-\int_{d_*}^N I_{0,1}I_{1,1}I_{2,j}\de u=O\!\left(\int_{d_*}^N\frac{\left(e^{i\lambda(j-2)\frac{\log u}{\log N}}-1\right)}{u^{j-1}\log u}\de u\right)=
O\!\left(\left.\tilde J_{1,1,j}(x)\right|_{\log d_*}^{\log N}
\right),\nonumber\\
J_{1,2,j}&=&-\int_{d_*}^N I_{0,1}I_{1,2}I_{2,j}\de u=O\!\left(\int_{d_*}^N\frac{\left(e^{i\lambda\frac{\log u}{\log N}}-1\right)\left(e^{i\lambda(j-2)\frac{\log u}{\log N}}-1\right)}{u^{j}\log u}\de u\right)=
O\!\left(\left.\tilde J_{1,2,j}(x)\right|_{\log d_*}^{\log N}
\right),\nonumber\\
J_{1,3,j}&=&-\int_{d_*}^N I_{0,1}I_{1,3}I_{2,j}\de u=O\!\left(\int_{d_*}^N\frac{\left(e^{i\lambda(j-2)\frac{\log u}{\log N}}-1\right)}{u^{j+1}\log u}\de u\right)=
O\!\left(\left.\tilde J_{1,3,j}(x)\right|_{\log d_*}^{\log N}
\right),\nonumber\\
J_{1,2,1}&=&-\int_{d_*}^N I_{0,1}I_{1,2}I_{2,1}\de u=O\!\left(\int_{d_*}^N\frac{\left(e^{i\lambda\frac{\log u}{\log N}}-1\right)^2}{u^2\log u}\de u\right)=
O\!\left(\left.\tilde J_{1,2,1}(x)\right|_{\log d_*}^{\log N}
\right),\nonumber
\eey
where 
\bey 
\tilde J_{1,1,j}(x)&=&\Gamma(0,(j-2)x)-\Gamma\!\left(0,(j-2)\left(1-i\varepsilon\right)x\right),\nonumber\\
\tilde J_{1,2,j}(x)&=&-\Gamma(0,(j-1)x)-\Gamma\!\left(0,(j-1)\left(1-i\varepsilon\right)x\right)+\nonumber\\
&&+\Gamma\!\left(0,\left(j-1-i\varepsilon\right)x\right)+\Gamma\!\left(0,\left(j-1-(j-2)i\varepsilon\right)x\right),\nonumber\\
\tilde J_{1,3,j}(x)&=&\Gamma(0,jx)-\Gamma\!\left(0,\left(j-(j-2)i\varepsilon\right)x\right),\nonumber\\
\tilde J_{1,2,1}(x)&=&-\Gamma(0,x)+2\Gamma\!\left(0,\left(1-i\varepsilon\right)x\right)-\Gamma\!\left(0,\left(1-2i\varepsilon\right)x\right).\nonumber
\eey

From (\ref{asymptotic-Gamma(0,w)}) and (\ref{asymptotic-Gamma(0,c+z)}) we get
\bey
J_{1,1,j}&=&O\!\left(\frac{1}{N^{j-2}\log N}\right)+O\!\left(\varepsilon\right),\nonumber\\
J_{1,2,j}&=&O\!\left(\frac{1}{N^{j-1}\log N}\right)+O\!\left(\varepsilon\right).\nonumber\\
J_{1,3,j}&=&O\!\left(\frac{1}{N^j\log N}\right)+O\!\left(\varepsilon\right).\nonumber\\
J_{1,2,1}&=&O\!\left(\frac{1}{N\log N}\right)+O\!\left(\varepsilon^2\right).\nonumber
\eey

The following error term estimates involve a combination of (\ref{property-Ei-0},\ref{property-Ei-infty}) and (\ref{asymptotic-Gamma(0,w)},\ref{asymptotic-Gamma(0,c+z)}).

\bey
J_{1,3,1}&=&-\int_{d_*}^N I_{0,1}I_{1,3}I_{2,1}\de u=O\!\left(\int_{d_*}^N\frac{\left(e^{i\lambda\frac{\log u}{\log N}}-1\right)}{u^3\log u}\de u\right)=
O\!\left(\left.\tilde J_{1,3,1}(x)\right|_{\log d_*}^{\log N}
\right),\nonumber\\
J_{2,2,1}&=&-\int_{d_*}^N I_{0,2}I_{1,2}I_{2,1}\de u=O\!\left(\int_{d_*}^N\frac{\left(e^{i\lambda\frac{\log u}{\log N}}-1\right)^2}{u^2 \log^2 u}\de u\right)=
O\!\left(\left.\tilde J_{2,2,1}(x)\right|_{\log d_*}^{\log N}
\right),\nonumber\\
J_{2,3,1}&=&-\int_{d_*}^N I_{0,2}I_{1,3}I_{2,1}\de u=O\!\left(\int_{d_*}^N\frac{\left(e^{i\lambda\frac{\log u}{\log N}}-1\right)}{u^3 \log^2 u}\de u\right)=
O\!\left(\left.\tilde J_{2,2,1}(x)\right|_{\log d_*}^{\log N}
\right),\nonumber
\eey
where
\bey
\tilde J_{1,3,1}&=&-\mathrm{Ei}\!\left(-2x\right)-\Gamma\!\left(0,\left(2-i\varepsilon\right)x\right),\nonumber\\
\tilde J_{2,2,1}&=& -\mathrm{Ei}(-x)+2\left(1-i\varepsilon\right)\Gamma\!\left(-1,\left(1-i\varepsilon\right)x\right)-
\left(1-2i\varepsilon\right)\Gamma\!\left(-1,\left(1-2i\varepsilon\right)x\right)-\frac{1}{x e^x},\nonumber\\
\tilde J_{2,3,1}&=&2\mathrm{Ei}\!\left(-2x\right)+\frac{1}{xe^{2x}}-\left(2-i\varepsilon\right)\Gamma\!\left(-1,\left(2-i\varepsilon\right)x\right).\nonumber
\eey
We get
\bey
J_{1,3,1}&=&O\!\left(\frac{1}{N\log N}\right)+O\!\left(\varepsilon\right),\nonumber\\
J_{1,2,2}&=&O\!\left(\frac{1}{N\log^2 N}\right)+O\!\left(\varepsilon^2\right),\nonumber\\
J_{2,3,1}&=&O\!\left(\frac{1}{N}\right)+O\!\left(\varepsilon\right).\nonumber
\eey

Combining all the error terms discussed in this appendix we obtain the desired estimate (\ref{J-allothers}).

\addcontentsline{toc}{section}{Bibliography}
\bibliographystyle{plain}
\bibliography{bibliography-NSF}
\end{document}